\documentclass[12pt,twoside]{article}
\usepackage{etex}
\usepackage[dvipsnames]{xcolor} 
\usepackage{color}
\usepackage{tcolorbox}
\usepackage{booktabs}
\usepackage{amsthm}
\usepackage{amsfonts,amssymb,amsxtra,url,float} 
\allowdisplaybreaks[4] 
\usepackage[colorlinks,
  linkcolor=magenta, %
  anchorcolor=Periwinkle,
  citecolor=red,
  urlcolor=blue
  ]{hyperref} 
\usepackage{enumitem}
\usepackage{geometry} 
\usepackage{rotating} 
\usepackage{lscape} 
\usepackage{multirow}
\usepackage{graphicx} 
\usepackage{subfigure} 
\usepackage{tikz}
\usepackage{pgfplots}
\usepackage{tikz-3dplot}
\usetikzlibrary{patterns}
\usetikzlibrary{3d,calc}
\usetikzlibrary{decorations.pathreplacing,decorations.markings}
 \tikzset{
  on each segment/.style={
    decorate,
    decoration={
      show path construction,
      moveto code={},
      lineto code={
        \path [#1]
        (\tikzinputsegmentfirst) -- (\tikzinputsegmentlast);
      },
      curveto code={
        \path [#1] (\tikzinputsegmentfirst)
        .. controls
        (\tikzinputsegmentsupporta) and (\tikzinputsegmentsupportb)
        ..
        (\tikzinputsegmentlast);
      },
      closepath code={
        \path [#1]
        (\tikzinputsegmentfirst) -- (\tikzinputsegmentlast);
      },
    },
  },
  mid arrow/.style={postaction={decorate,decoration={
        markings,
        mark=at position 0.6 with {\arrow[#1]{stealth}} 
      }}},
}
\usetikzlibrary{arrows}
\usetikzlibrary{trees}

\usetikzlibrary{matrix}
\usetikzlibrary{patterns}
\usetikzlibrary{shadings} 
\usepackage{fancyhdr} 
\def\headertitle{Kleisli convolution representations and a Margolis--Sakurai version of MIP}
\def\fstpage{1} 
\def\page{$\begin{matrix} {\color{white}0} \\ \thepage \end{matrix}$} 

\usepackage[all]{xy} 
\usepackage{dsfont} 
\usepackage{cite}
\usepackage{mathrsfs} 
\numberwithin{figure}{section}
\usepackage{marginnote} 
\usepackage{graphicx} 
\usepackage{multicol} 

\usepackage{enumitem}
\setenumerate[1]{itemsep=0pt,partopsep=0pt,parsep=\parskip,topsep=3pt}
\setitemize[1]{itemsep=0pt,partopsep=0pt,parsep=\parskip,topsep=3pt}
\setdescription{itemsep=0pt,partopsep=0pt,parsep=\parskip,topsep=3pt}
\setlist[itemize]{leftmargin=35pt}
\setlist[enumerate]{leftmargin=35pt}
\usepackage{changepage} 
\newcommand{\checks}[1]{{\color{black}{#1}}} 

\usepackage{contour} 
\usepackage{xcolor}
\contourlength{0.03em}
\contournumber{auto}

\newtheorem{theoremAlph}{Theorem}

\newtheorem{theorem}{Theorem}[section]
\newtheorem{lemma}[theorem]{Lemma}
\newtheorem{corollary}[theorem]{Corollary}
\newtheorem{main theorem}[theorem]{Main Theorem}
\newtheorem{proposition}[theorem]{Proposition}
\newtheorem{definition}[theorem]{Definition}

\newtheorem{remark}[theorem]{Remark}
\newtheorem{example}[theorem]{Example}

\newtheorem{question}[theorem]{Question}
\newtheorem{fact}[theorem]{Fact}

\numberwithin{equation}{section}

\def\orcid{
\begin{tikzpicture}[baseline=-1mm]
\filldraw[Green!35] (0,0) circle (5pt);
\filldraw[white] (0,0) node{\tiny\textbf{iD}};
\end{tikzpicture}
}

\def\=<{\leqslant} 
\def\>={\geqslant} 
\def\NN{\mathbb{N}} 
\def\ZZ{\mathbb{Z}} 

\newcommand{\modcat}{\mathsf{mod}}

\newcommand{\kk}{\mathds{k}} 
\newcommand{\Hom}{\mathrm{Hom}} %
\newcommand{\End}{\mathrm{End}} %
\newcommand{\Aut}{\mathrm{Aut}} %
\newcommand{\compos}{\begin{smallmatrix}\circ\end{smallmatrix}}

\usepackage{extarrows}
\def\To{\xlongrightarrow}

\def\calC{\mathcal{C}}
\def\calD{\mathcal{D}}

\def\calP{\mathcal{P}}

\def\calT{\mathcal{T}}

\def\EE{\mathbb{E}}
\def\FF{\mathbb{F}}

\def\orcid{
\begin{tikzpicture}[baseline=-1mm]
\filldraw[Green!35] (0,0) circle (5pt);
\filldraw[white] (0,0) node{\tiny\textbf{iD}};
\end{tikzpicture}
}

\title{\bf
Kleisli convolution representations and the Margolis--Sakurai version of the modular isomorphism problem
}

\vspace{5mm}

\author{
Yu-Zhe Liu$^{\ref{Author1}, \ref{Author2}, \href{https://orcid.org/0009-0005-1110-386X}{\orcid}\ref{orcid1}}$
}
\date{ }
\begin{document}






\maketitle

\begin{enumerate}[label=\textbf{\color{red}\arabic*}] \footnotesize
  \item
    \begin{center}
      School of Mathematics and Statistics,
      Guizhou University, Guiyang 550025, Guizhou, China;
    \end{center}
    \label{Author1}

  \item
    \begin{center}
      Guizhou Provincial Key Laboratory of Applied Mathematics and Computing Power \& Algorithms,
      Guiyang, 550025, Guizhou, China;
    \end{center}
    \label{Author2}
%
%

  \item[]
    \begin{center}
      E-mail:
      \url{liuyz@gzu.edu.cn} / \url{yzliu3@163.com}.
    \end{center}
\end{enumerate}
\vspace{1mm}
\begin{enumerate}[label=\textbf{\color{red}$\ddag$}]
  \item \footnotesize
    \begin{center}
      Corresponding author
    \end{center} \label{CorrespondingAuthor}
\end{enumerate}
\vspace{1mm}
\begin{enumerate}[leftmargin=6.5cm] \footnotesize
  \item[\orcid]
      ORCID: \href{https://orcid.org/0009-0005-1110-386X}{0009-0005-1110-386X}
      \label{orcid1} 
\end{enumerate}

\hspace{2mm}



\begin{adjustwidth}{1cm}{1cm} 
\noindent \footnotesize 
\textbf{Abstract.}
We introduce Kleisli convolution representations for groups, rings, and algebras. We show that their representation categories are equivalent to the usual ones for groups and finite-dimensional algebras, but for rings only recover modules whose underlying Abelian groups are free of finite rank. 
We also apply the Kleisli convolution representation to provide a partial answer to the Margolis--Sakurai version of the modular isomorphism problem:
If $G$ is a finite $p$-group with $D_3(G)=1$ and $\gcd(m,c_G!)=1$, then $\mathbb F_{p^m}G\cong\mathbb F_{p^m}H$ implies $G\cong H$, where $c_G$ is the number of geometric connected components of $\mathrm{Aut}(\overline{\mathbb F}_pG)$.

\noindent\textbf{Key words.}
Monoidal monad; free semimodule; powerset monad;
Eilenberg--Moore category; module category

\noindent\textbf{2010 Mathematics Subject Classification.}
18C20 
18D10 
16G10 
16D10 
20C15 

\end{adjustwidth}

%
%
%


\newpage


\setcounter{tocdepth}{2}
\setcounter{secnumdepth}{3}
\tableofcontents


\def\gldim{\mathrm{gl.dim}}
\def\perm{\mathrm{perm}}
\def\forb{\mathrm{forb}}
\def\rmv{\mathsf{v}}
\def\rmp{\mathsf{w}}

\newpage

\section{Introduction}

\def\id{\mathrm{id}}
\def\bfone{\mathbf{1}}
\def\stmap{\mathrm{st}}
\def\Kl{\mathsf{Kl}}
\def\scrT{\mathscr{T}}
\def\scrM{\mathscr{M}}
\def\scrP{\mathscr{P}}
\def\semimod{\mathcal{M}}

\def\Set{\mathsf{Set}}
\def\FinSet{\mathsf{FinSet}}
\def\Ab{\mathsf{Ab}}
\def\Mod{\mathsf{Mod}} \def\Modcat{\mathsf{Mod}}
\def\Mat{\mathsf{Mat}}
\def\rep{\mathsf{rep}}
\def\Pplus{\mathcal{P}_+}
\def\KCrep{\mathsf{Klcrep}}
\def\Perm{\mathsf{Perm}}
\def\supp{\mathrm{supp}}
\def\Klcorresp{\mathcal{K}\!\ell}
\def\KlHom{\mathds{H}}

\newcommand{\defines}[1]{{\it\color{violet}#1}}

%

\subsection{Motivations}

\paragraph{Kleisli convolution representations of power monoids}
Representation theory is one of the basic branches of algebra.
It studies algebraic systems by describing their elements as transformations of sets, Abelian groups or vector spaces.
For example, a permutation representation of a group $G$ is given by a group homomorphism from $G$ to a symmetric group,
while a linear representation of $G$ over a field $\kk$ is given by a group homomorphism from $G$ to a general linear group.
Similarly, a left module over a ring $R$ can be described by a ring homomorphism from $R$ to an endomorphism ring.
After fixing a basis, the endomorphisms of a finite-dimensional vector space can be represented by matrices.
We refer to \cite{BiggsWhite1979,EGHLSVY2011,Lam,ARS} for the basic theory of these representations.


The notion of a Kleisli category was introduced by Kleisli in \cite{Kleisli}.
Every monad, written as $\scrT=(T,\eta,\mu)$, on a category $\calC$ gives a Kleisli category $\Kl(\scrT)$
whose objects are the objects of $\calC$ and whose morphisms from $X$ to $Y$
are the morphisms from $X$ to $T(Y)$ in $\calC$ \cite{Kleisli, Mac1998, Manes1976, Riehl}.
Such a morphism is called a Kleisli arrow.
If $\scrT$ is a monoidal monad on a monoidal category $\calC$,
then the monoidal structure of $\calC$ induces a monoidal structure on $\Kl(\scrT)$ \cite{Kock1972, Zaw2012}.
On the other hand, for a comonoid object $X$ and a monoid object $Y$ in a monoidal category,
the Hom-space from $X$ to $Y$ is a monoid under the usual convolution given in \cite[Definition 1.13]{AM2010}.
Combining these two constructions, we obtain a convolution on $\Hom_{\Kl(\scrT)}(X,Y)$
whose multiplication involves both the tensor product of Kleisli arrows and the Kleisli composition.
We call this multiplication the Kleisli convolution.
Its precise construction is given in Theorem \ref{thm:Kl convol-semigp}.

The motivation of the present paper comes from the previous work in \cite{WHL2026}.
In that paper, the authors observed that Kleisli categories can also be used to describe power semigroups by Kleisli convolution,
and showed that power semigroups of groups and reduced finitary power monoids can be realized as convolution monoids in Kleisli categories of powerset monads.
This naturally leads us to ask the following question.

\begin{question}
Can the Kleisli convolution representation be used to give representations of other algebraic systems, such as groups, rings and algebras?
\end{question}

\noindent
To carry out this idea, let $S$ be a commutative semiring and let $\scrM_S$ be the free $S$-semimodule monad on $\Set$.
This monad and its relation with matrix categories are well-known, see \cite[Sections 5.1 and 6.1]{CJ2013}.
In particular, when restricted to finite sets,
the Kleisli category $\Kl(\scrM_S)$ can be identified with the category $\Mat_S$ of matrices over $S$.
Proposition \ref{prop:matrix-convolution} gives a bijection
$\Phi_X: \Hom_{\Kl(\scrM_S)}(\bfone,E_X) \xrightarrow{\cong} \Mat_X(S)$
which implies that $\Hom_{\Kl(\scrM_S)}(\bfone,E_X)$ is an $S$-algebra.
This construction allows us to define Kleisli convolution representations of groups, rings and algebras,
see Definitions \ref{def:Kl convol rep gp}, \ref{def:Kl convol rep ring}, \ref{def:Kl convol rep algebra}.

\paragraph{Modular isomorphism problem}

Another motivation of this paper is the modular isomorphism problem (MIP) and its dependence on the coefficient field.
To be more precise, the following question.

\begin{question}\label{quest:MIP}
For finite $p$-groups $G$ and $H$, and a field $\FF$ of characteristic $p$,
MIP asks whether an isomorphism $\FF G\cong\FF H$ of
$\FF$-algebras implies $G\cong H$.
\end{question}

For the history and methods of this problem, we refer to \cite{Margolis2022}.
Although many positive results have been obtained, the problem has a negative answer in general.
Garc\'ia-Lucas, Margolis and del R\'io constructed non-isomorphic finite $2$-groups whose group algebras are isomorphic over every field of characteristic $2$ \cite{GLMR2022}.
These examples show that restrictions on the groups are necessary.
They do not, however, answer whether a positive result for a fixed group over the prime field remains valid
over every field of the same characteristic, see \cite{MS2025}.

One classical method to the problem is to recover information about dimension subgroups from the group algebra.
In 1972, Passi and Sehgal proved that $\FF_pG$ determines the isomorphism type of $D_n(G)/D_{n+2}(G)$ for every positive integer $n$,
where $D_n(G)$ denotes the $n$-th dimension subgroup, i.e.,
\[ D_n(G):=\{g\in G:g-1\in I_G^n\} =G\cap(1+I_G^n),\]
where $I_G:= \big\{\sum\limits_{g\in G}a_gg\in\FF_pG: \sum\limits_{g\in G}a_g=0\big\}$
is an ideal of $\FF_p G$, see \cite{PSl1972}.
Taking $n=1$, their result shows that a finite $p$-group with $D_3(G)=1$ is determined by its group algebra over $\FF_p$.
This result was later strengthened to the determination of $D_n(G)/D_{2n+1}(G)$ over $\FF_p$, see \cite{Furukawa1981,RitterSehgal1983}.
However, these results over the prime field do not automatically give the same conclusions over ``larger'' fields.
As explained in \cite[Section 2.2]{MS2025}, the arguments use properties of coefficients in $\FF_p$
which are not available for arbitrary coefficients.
This leads Margolis and Sakurai to ask:

\begin{question}[{\!\cite[Question 2.10]{MS2025}}] \label{quest:MIP MS2025}
Whether $\FF G\cong\FF H$ implies $G\cong H$ for every field $\FF$ of characteristic $p$ when $D_3(G)=1$.
\end{question}

The role of finite extensions is clarified by a reduction theorem of Garc\'ia-Lucas and del R\'io \cite{GLR2024}.
They proved that an isomorphism of group algebras over a field implies an isomorphism over some finite extension of its prime field.
In particular, for the MIP, it is enough to consider finite coefficient fields.
This reduction does not assert that an isomorphism over a finite extension descends to the prime field.
Motivated by this remaining difficulty, we realize group algebras as regular Kleisli convolution subalgebras
and apply Lang's theorem \cite{Lang1956} to obtain a sufficient condition for such descent.
Combining this condition with the result of Passi and Sehgal, we obtain a positive answer to Question \ref{quest:MIP MS2025} for finite extensions $\FF_{p^m}$ satisfying $\gcd(m,c_G!)=1$.
Here $c_G$ is the number of geometric connected components of the automorphism algebraic group of $\FF_pG$.
The allowed extension degrees therefore depend on $G$.

\subsection{Outline}

The theoretical outline and the main results of this paper are as follows.

First, we consider the Kleisli convolution in a general Kleisli category.
Let $\scrT=(T,\eta,\mu)$ be a monoidal monad on a monoidal category $\calC$.
Then, for a comonoid object $X$ and a monoid object $Y$ in $\Kl(\scrT)$,
the Kleisli Hom-space $\Hom_{\Kl(\scrT)}(X,Y)$ is a monoid under the Kleisli convolution.
This result provides the categorical construction used throughout this paper
(note that this fact has been proved in \cite{Zaw2012}, but no Kleisli category was used as a descriptive tool).

Second, we apply the powerset monad and the free $\kk$-module monad to representations of groups.
We first define finite Kleisli convolution representations by using the powerset monad.
We then define their linear version by using the free $\kk$-module monad.
Their relations with the usual representation categories are given by the following theorem.

\begin{theoremAlph}\label{thm:main-group-KCrep}
Let $G$ be a group and let $\kk$ be a field. The following statements
hold.
\begin{enumerate}[label={\rm(\arabic*)}]
  \item {\rm(Theorems \ref{thm:KLcrep gp} and \ref{thm:KCrep-perm rep})}
    Finite Kleisli convolution representations of $G$ with respect to the powerset monad form a category $\KCrep(G)$,
    and there is an equivalence
    \[\KCrep(G)\simeq\Perm(G). \]
    where $\Perm(G)$ is the category of finite permutation representations of $G$.

  \item {\rm(Theorems \ref{thm:linear KLcrep gp} and \ref{theo:KC-linear})}
    Linear Kleisli convolution representations of $G$ over $\kk$ form a category $\KCrep_{\kk}(G)$,
    and there is a $\kk$-linear equivalence
    \[ \KCrep_{\kk}(G)\simeq\rep_{\kk}(G).\]
\end{enumerate}
\end{theoremAlph}

\noindent
Therefore, the Kleisli convolution construction provides categorical descriptions of
both finite permutation representations and finite-dimensional linear representations of groups.

Third, we consider an arbitrary associative ring $R$ with identity.
In this case, we apply the free Abelian group monad.
Since every object obtained from this monad \checks{by applying it to a finite set} has an underlying Abelian group which is free of finite rank, a freeness restriction appears.
The main results for rings are as follows.

\begin{theoremAlph}\label{thm:main-ring-KCrep}
Let $R$ be an associative ring with identity. The following statements hold.
\begin{enumerate}[label={\rm(\arabic*)}]
  \item {\rm(Theorems \ref{thm:ring-KCrep-category} and
  \ref{thm:ring-free})}
  Kleisli convolution representations of $R$ form an additive category $\KCrep_{\ZZ}(R)$,
  and there is an additive equivalence
  \[ \KCrep_{\ZZ}(R) \simeq R\text{-}\mathsf{free}_{\ZZ}, \]
  where $R\text{-}\mathsf{free}_{\ZZ}$ is the full subcategory of finitely generated left $R$-modules whose underlying Abelian groups are free of finite rank.

  \item {\rm(Corollary \ref{coro:em-ring})}
  Let $\scrT_R=(R\otimes_{\ZZ}-,\eta,\mu)$ be the corresponding monad on $\Ab$
  and $\Ab^{\scrT_R}$ be the Eilenberg--Moore category induced by $\scrT_R$.
  Then there is a fully faithful functor
  \[ \KCrep_{\ZZ}(R) \to \Ab^{\scrT_R} ~ (\simeq R\text{-}\Mod). \]
\end{enumerate}
\end{theoremAlph}

\noindent
Theorem \ref{thm:main-ring-KCrep} also reveals a surprising difference among groups, rings and algebras. Recall that every ring is an Abelian group with respect to addition, and every $\kk$-algebra is in particular a ring. Thus, one may expect that their Kleisli convolution representation categories behave in a similar way.
However, \textbf{the Kleisli convolution representation categories of groups and finite-dimensional $\kk$-algebras are equivalent to their usual representation categories, while $\KCrep_{\ZZ}(R)$ is generally not equivalent to $R\text{-}\Mod$}.
It only recovers the full subcategory $R\text{-}\mathsf{free}_{\ZZ}$. This unexpected difference is one of the main observations of this paper.

Fourth, we consider a finite-dimensional algebra $A$ over a field $\kk$.
Since every finite-dimensional $\kk$-vector space is free,
the restriction appearing in Theorem \ref{thm:main-ring-KCrep} does not occur.
Therefore, the whole finite-dimensional representation category can be recovered.

\begin{theoremAlph}[{\rm Theorems \ref{thm:algebra-KCrep-category}, \ref{thm:algebra-equivalence}}]
\label{thm:main-algebra-KCrep}
Let $A$ be a finite-dimensional $\kk$-algebra. Then Kleisli convolution representations of $A$ form a $\kk$-linear category $\KCrep(A)$, and we have a $\kk$-linear equivalence
\[ \KCrep(A)\simeq\rep(A).\]
\end{theoremAlph}

We also study some special Kleisli convolution representations.
For a ring $R$, Proposition \ref{prop:ring-flat} describes flat Kleisli convolution representations;
Propositions \ref{prop:ring-injective-zero} and \ref{prop:ring-simple-zero} show that there are no nonzero injective Kleisli convolution representations and no simple Kleisli convolution representations of $R$.
For a finite-dimensional algebra $A$, Propositions \ref{prop:algebra-projective}, \ref{prop:algebra-injective} and \ref{prop:algebra-simple} show that projective, injective and simple Kleisli convolution representations correspond to projective, injective and simple finite-dimensional $A$-modules, respectively.

The first purpose of this paper is to provide a common Kleisli convolution description for these familiar representation categories.
Theorems \ref{thm:main-group-KCrep}, \ref{thm:main-ring-KCrep} and \ref{thm:main-algebra-KCrep} also show the scope of this description.
Over a field, finite-dimensional representations can be completely recovered from the matrix Kleisli category. For a general ring, only the modules whose underlying Abelian groups are free of finite rank are obtained, and the Eilenberg--Moore category is needed in order to include all modules.

The second purpose of this paper is to study MIP.
We give a partial answer to Question \ref{quest:MIP MS2025}.
For a finite $p$-group $G$ with $D_3(G)=1$, we show that $\FF_{p^m}G\cong\FF_{p^m}H$ implies $G\cong H$ whenever $\gcd(m,c_G!)=1$.
This is proved by realizing $\FF_pG$ as a regular Kleisli convolution subalgebra and applying Lang's theorem.
See the following result.

\begin{theoremAlph}[{Theorem \ref{coro:third dimension finite extensions}}]
Let $G$ be a finite $p$-group with $D_3(G)=1$.
Let $m$ be a positive integer such that $\gcd(m,c_G!)=1$.
Then, for every finite $p$-group $H$,
\[  \FF_{p^m}G\cong\FF_{p^m}H \text{~implies~} G\cong H. \]
\end{theoremAlph}

\subsection{Organization}

The paper is organized as follows. In Section \ref{sect:prelim}, we recall monads, Kleisli categories, monoidal categories and convolution monoids, and establish the general Kleisli convolution construction.
In Section \ref{sect:free semimod monad}, we introduce the free semimodule monad and describe matrix multiplication as a Kleisli convolution.
In Section \ref{sect:group KCR}, we study finite and linear Kleisli convolution representations of groups. In Section \ref{sect:ring KCR}, we study Kleisli convolution representations of rings and compare them with the corresponding Eilenberg--Moore category.
In Section \ref{sect:algebra KCR}, we study Kleisli convolution representations of finite-dimensional algebras.
We also consider some special Kleisli convolution representations in Sections \ref{sect:ring KCR} and \ref{sect:algebra KCR}.
In Section \ref{sect:MIP}, we study Question \ref{quest:MIP MS2025} and provide a positive answer for finite extensions $\FF_{p^m}$ satisfying $\gcd(m,c_G!)=1$.

\section{Preliminaries} \label{sect:prelim}

\textsl{In this section, we recall some basic notions concerning monads, Kleisli categories and monoidal categories. We also introduce the Kleisli convolution and describe its monoid structure.}

\subsection{Kleisli categories} \label{subset:Kl cat}

Let $\calC$ be a category. Recall that a \defines{monad} on $\calC$ is a triple $\scrT = (T,\eta,\mu)$,
where $T: \calC\to \calC$ is a functor and $\eta: 1_{\calC} \to T$ and $\mu: T^2\to T$ are natural transformations
satisfying the usual \defines{unit condition} and \defines{associativity condition}. To be precise, for every object $X$, we have
$\mu_X \compos T(\eta_X)=\id_{TX}=\mu_X\compos\eta_{TX}$ and $\mu_X \compos T(\mu_X)=\mu_X \compos \mu_{TX}$.

\begin{definition}\label{def:Kleisli cat} \rm
The Kleisli category of a monad $\scrT = (T, \eta, \mu)$ defined on a category $\calC$,
denoted by $\Kl(\scrT)$, has the same objects as $\mathcal C$, and $\Hom_{\Kl(\scrT)}(X,Y)=\Hom_{\calC}(X,TY)$.
Here, each homomorphism in $\Hom_{\Kl(\scrT)}(X,Y)$ is written as $X \rightsquigarrow Y$,
and we call it a \defines{Kleisli arrow}.
For two Kleisli arrows $f: X\to TY$ and $g: Y\to TZ$, their composition, say \defines{Kleisli composition}, is defined as
\[ g \compos_{\Kl} f=\mu_Z\compos T(g)\compos f: X\to TZ. \]
\end{definition}

\begin{remark}\rm
One can check that $\eta_X: X\to TX$, as a Kleisli arrow in $\Hom_{\Kl(\scrT)}(X,X)$, is the identity arrow $\id_X$ of $X$,
i.e., \[ \id_X = \eta_X: X \rightsquigarrow X ;\]
and for three Kleisli arrows $f: X \rightsquigarrow Y$, $g: Y \rightsquigarrow Z$ and $h: Z \rightsquigarrow W$,
we have
\[ h\compos_{\Kl}(g\compos_{\Kl} f)=(h\compos_{\Kl} g)\compos_{\Kl} f. \]
The two unit conditions of the monad similarly give
\[f\compos_{\Kl}\eta_X=f=\eta_Y\compos_{\Kl} f.\]
Thus, the preceding data given in Definition \ref{def:Kleisli cat} indeed form a category.
\end{remark}

\begin{definition}[{\!\!\cite{Ben1963,Mac1963}}] \rm
A \defines{monoidal category} consists of a category $\calC$, a bifunctor
$\otimes:\calC\times\calC\to\calC$ called the \defines{tensor product},
an object $\bfone$ called the \defines{unit},
and three natural isomorphisms
\[
\alpha_{X,Y,Z}:(X\otimes Y)\otimes Z \xrightarrow{\sim} X\otimes(Y\otimes Z),
\quad
\lambda_X: \bfone \otimes X\xrightarrow{\sim}X,
\quad
\rho_X:X\otimes \bfone \xrightarrow{\sim}X,
\]
called the \defines{associator}, \defines{left unitor}, and \defines{right unitor} respectively.
These are required to satisfy the triangle and pentagon identities,
which ensure that the tensor product is associative and unital in a coherent manner.
\begin{itemize}
  \item Triangle identity: the following diagram
\[\xymatrix{
& X \otimes Y
& \\
  X \otimes (\bfone \otimes Y)
  \ar@{<-}[rr]_{\alpha_{X,\bfone,Y}}
  \ar[ru]^{\id_X \otimes \lambda_Y}
&
& (X \otimes \bfone) \otimes Y
\ar[lu]_{\rho_X \otimes \id_Y} }\]
   commutes.

  \item Pentagon identity: the following diagram
\[\xymatrix{
   ((X\otimes Y)\otimes Z)\otimes W
   \ar[d]_{\alpha_{X,Y,Z}\otimes\id_W}
   \ar[rr]^{\alpha_{X\otimes Y, Z, W}}
 &
 & (X\otimes Y)\otimes(Z\otimes W)
   \ar[d]^{\alpha_{X, Y, Z\otimes W}}
\\
   (X\otimes (Y\otimes Z))\otimes W
   \ar[rd]_{\alpha_{X,Y\otimes Z, W}}
 &
 & X\otimes (Y\otimes (Z\otimes W))
\\
 & X\otimes ((Y\otimes Z)\otimes W) \ar[ru]_{\id_X\otimes\alpha_{Y,Z,W}}
 &
}
\]
  commutes.
\end{itemize}
\end{definition}

Suppose that $\calC$ is a monoidal category and that $T$ is a commutative monoidal monad.
Then $\Kl(\scrT)$ is again a monoidal category. In particular, there are natural arrows
\[ \vartheta_{X,Y}: TX\otimes TY\longrightarrow T(X\otimes Y)\]
which make it possible to take the tensor product
\begin{align}\label{eq:Kleisli tensor}
  f \otimes_{\Kl} g := \vartheta_{Y,Y} \compos (f\otimes g) : X\otimes X \to T(Y\otimes Y)
\end{align}
of Kleisli arrows $f: X \rightsquigarrow Y$ and $g: X\rightsquigarrow Y$.
For monoidal monads and the induced monoidal structures,
we refer to \cite{Kock1972,Zaw2012}.

The following lemma is well-known.

\begin{lemma}\label{lemm:Klcorresp}
Let $\scrT=(T,\eta,\mu)$ be a monad on $\Set$, and let $X$ be an arbitrary set and $\bfone=\{\bullet\}$ be a singleton set.
Then there is a bijection
\[ \Klcorresp_X: \Hom_{\Kl(\scrT)}(\bfone, X) \to T(X) \]
sending each Kleisli arrow $f: \bfone \rightsquigarrow X$ to the element $f(\bullet)$ of $TX$.
\end{lemma}

\begin{proof}
Write the Kleisli arrow $f:\bfone\rightsquigarrow X$ as a map $f:\bfone\to T(X)$ in $\Set$.
Since $\bfone$ has only one element $\bullet$, the map $f$ is uniquely determined by the element $f(\bullet)\in T(X)$.
Let $f,~g:\bfone\rightsquigarrow X$ be two Kleisli arrows satisfying
$\Klcorresp_X(f)=\Klcorresp_X(g)$. Then we have $f(\bullet)=g(\bullet)$.
Since $\bullet$ is the only element of $\bfone$, it follows that $f=g$.
Thus, $\Klcorresp_X$ is injective.
On the other hand, let $A\in T(X)$. Consider the map $f_A:\bfone\to T(X)$ by $f_A(\bullet)=A$.
Then $f_A$ is a Kleisli arrow from $\bfone$ to $X$, and $\Klcorresp_X(f_A)=A$.
Therefore, $\Klcorresp_X$ is a bijection.
\end{proof}

\subsection{Convolution monoids}

First of all, we recall the convolution construction.

\begin{definition}\rm
Let $(\calD,\otimes, \bfone)$ be a monoidal category.
\begin{enumerate}
\item
A \defines{comonoid object} in $\calD$ is a triple $(B,\Delta,\varepsilon)$ consisting of an object $B$, a morphism $\Delta:B\to B\otimes B$ which is called a \defines{comultiplication}, and a morphism $\varepsilon:B\to \bfone$ which is called a \defines{counit}, such that the following diagrams
\[
\begin{gathered}
\xymatrix@C=2cm{
  B \ar[r]^{\Delta} \ar[d]_{\Delta} & B\otimes B \ar[d]^{\Delta\otimes\id_B} \\
  B\otimes B \ar[r]_{\id_B\otimes\Delta} & B\otimes B\otimes B
}
\quad
\xymatrix{
  \bfone \otimes B
  \ar@{<-}[r]^{\varepsilon \otimes\id_B}
  \ar@{<-}[dr]_{\lambda_B^{-1}} & B\otimes B
  \ar@{<-}[d]^{\Delta} & B \otimes \bfone
  \ar@{<-}[l]_{\id_B\otimes \varepsilon }
  \ar@{<-}[dl]^{\rho_B^{-1}} \\
  & B &
}
\end{gathered}
\]
commute, i.e., $\Delta$ is coassociative and $\varepsilon$ is a counit.
It is \defines{cocommutative} if $\Delta = \sigma_{B,B}\compos \Delta$, where $\sigma$ is the symmetry isomorphism of the monoidal category (when it exists).

  \item
A \defines{monoid object} in $\calD$ is a triple $(H,m,u)$ consisting of an object $H$, a morphism $m:H\otimes H\to H$ which is called a \defines{multiplication}, and a morphism $u:\bfone\to H$ which is called a \defines{unit}, such that the following diagrams commute:
\[
\begin{gathered}
\xymatrix@C=1.65cm{
  H\otimes H\otimes H \ar[r]^{\id_H\otimes m} \ar[d]_{m\otimes\id_H} & H\otimes H \ar[d]^{m} \\
  H\otimes H \ar[r]_{m} & H
}
\quad
\xymatrix{
  \bfone\otimes H
  \ar[r]^{u\otimes\id_H}
  \ar[dr]_{\lambda_H}
& H\otimes H \ar[d]^{m}
& H\otimes \bfone
  \ar[l]_{\id_H\otimes u}
  \ar[dl]^{\rho_H} \\
  & H &
}
\end{gathered}
\]
i.e., $m$ is associative and $u$ is a two-sided unit.
It is \defines{commutative} if $m = m\compos\sigma_{H,H}$.
\end{enumerate}
\end{definition}

Let $(X,\Delta,\varepsilon)$ be a comonoid object and let $(Y,m,u)$ be a monoid object in $(\calD,\otimes, \bfone)$.
For $f,g\in\Hom_{\calD}(X,Y)$, their \defines{convolution} is defined as
\[ f \ast g = m\compos(f\otimes g) \compos \Delta : \quad X \To{\Delta} X\otimes X \To{f\otimes g} Y \otimes Y \To{m} Y.\]
Then one can check that
\[ f \ast (g \ast h) = (f\ast g) \ast h \]
for arbitrary morphisms $f$, $g$, $h$ $\in \Hom_{\calD}(X,Y)$.
Therefore, $\Hom_{\calD}(X,Y)$ is a semigroup.
Furthermore, it is a monoid since it has a unit $u\compos\varepsilon$.
This idea can be applied to Kleisli categories, and we obtain the following lemma.


In \cite[Definition 1.13]{AM2010}, Aguiar and Mahajan provided a method to construct a monoid by using a comonoid object and a monoid object. This construction is a standard convolution-monoid construction in a monoidal category.
In \cite{Zaw2012}, Zawadowski proved that $\Kl(\scrT)$ has a monoidal structure.
Thus, we immediately have the following result.

\begin{theorem}\label{thm:Kl convol-semigp}
Let $\calC$ be a monoidal category, $T=(T,\eta,\mu)$ be a monoidal monad, and $\Kl(\scrT)$ be the Kleisli category given by $T$.
Then for any comonoid object $(X,\Delta,\varepsilon)$ and any monoid object $(Y,m,u)$
{\rm(}$\Delta$, $\varepsilon$, $m$, and $u$ are Kleisli arrows{\rm)},
the Kleisli Hom-space $\Hom_{\Kl(\scrT)}(X,Y)$ is a monoid.
\end{theorem}

For the reader's ease, we include a proof of the above theorem. Since it is formulated in terms of the Kleisli category, it may look different from the treatments in \cite{AM2010, Zaw2012}.
Nevertheless, the proof is conceptually identical, for we only check, directly from the definition, that $\Hom_{\Kl(\scrT)}(X,Y)$ is a monoid.

\begin{proof}
We prove that the convolution is associative and has a unit.
Let $f$ and $g$ be two Kleisli arrows in $\Hom_{\Kl(\scrT)}(X,Y)$.
Then they are morphisms $f,g:X \to TY$ in $\calC$.
By \eqref{eq:Kleisli tensor}, the tensor product of the corresponding Kleisli arrows is represented in $\calC$ by
\[ \vartheta_{Y,Y}\compos(f\otimes g): X\otimes X\to T(Y\otimes Y).\]
Therefore, their Kleisli convolution is the following morphism in $\calC$:
\begin{align} \label{eq:Kl convol-semigp 1}
f\ast g =
\mu_Y \compos T(m) \compos
\mu_{Y\otimes Y} \compos
T(\vartheta_{Y,Y}) \compos
T(f\otimes g) \compos \Delta: X \to TY.
\end{align}

We first prove the associativity. Let $f,g,h \in \Hom_{\Kl(\scrT)}(X,Y)$.
By \eqref{eq:Kl convol-semigp 1}, and using the naturality and associativity conditions of the monoidal structure $\vartheta$,
we obtain
\begin{align}\label{eq:Kl convol-semigp 2}
(f\ast g)\ast h ={} &
 \mu_Y \compos T(\mu_Y\compos T(m)\compos \vartheta_{Y,Y}\compos(m\otimes\eta_Y)) \nonumber \\
\compos & \mu_{Y\otimes Y\otimes Y} \nonumber  \\
\compos & T( \vartheta_{Y\otimes Y,Y}
\compos (\vartheta_{Y,Y}\otimes\id_{TY})
\compos ((f\otimes g)\otimes h)) \nonumber \\
\compos & \mu_{X\otimes X\otimes X} \nonumber \\
\compos & T( \vartheta_{X\otimes X,X} \compos(\Delta\otimes\eta_X))
\compos \Delta.
\end{align}
Since $(Y,m,u)$ is a monoid object in $\Kl(\scrT)$, we have the following diagram
\[ \xymatrix@C=1.65cm{
  Y\otimes Y\otimes Y \ar@{~>}[r]^{\eta_Y\otimes m} \ar@{~>}[d]_{m\otimes\eta_Y} & Y\otimes Y \ar@{~>}[d]^{m} \\
  Y \otimes Y \ar@{~>}[r]_{m} & Y
} \]
commutes in the Kleisli category $\Kl(\scrT)$. The diagram as above can be translated to the following commutative diagram
\[ \xymatrix{
  Y\otimes Y\otimes Y
  \ar[r]_{\eta_Y\otimes m}
  \ar@/^1.5pc/[rr]^{\eta_Y\otimes_{\Kl} m}
  \ar[d]^{m\otimes\eta_Y}
  \ar@/_2.5pc/[dd]_{m\otimes_{\Kl}\eta_Y}
& TY \otimes TY
  \ar[r]_{\vartheta_{Y,Y}}
& T(Y\otimes Y) \ar[d]^{T(m)}
\\
  TY\otimes TY
  \ar[d]^{\vartheta_{Y,Y}}
& & TTY \ar[d]^{\mu_Y}
\\
  T(Y\otimes Y) \ar[r]_{T(m)}
& TTY \ar[r]_{\mu_Y}
& TY
}
\]
in the category $\calC$. Thus, in $\calC$, we have
\begin{align}\label{eq:Kl convol-semigp 3}
  \mu_Y\compos T(m)\compos
  \vartheta_{Y,Y}\compos(m\otimes\eta_Y)
= \mu_Y\compos T(m)\compos
  \vartheta_{Y,Y}\compos(\eta_Y\otimes m).
\end{align}
Similarly, since $(X,\Delta,\varepsilon)$ is a comonoid object in $\Kl(\scrT)$,
we have the following diagram
\[\xymatrix@C=1.65cm{
  X \ar@{~>}[r]^{\Delta} \ar@{~>}[d]_{\Delta} & X\otimes X \ar@{~>}[d]^{\Delta\otimes\eta_X} \\
  X\otimes X \ar@{~>}[r]_{\eta_X\otimes\Delta} & X\otimes X\otimes X
}\]
commutes in the Kleisli category $\Kl(\scrT)$. The diagram as above can be translated to the following commutative diagram
\[\xymatrix{
  X \ar[r]^{\Delta} \ar[d]_{\Delta}
& T(X\otimes X)
  \ar[rd]^{T(\Delta\otimes\eta_X)}
& &
\\
  T(X\otimes X)
  \ar[rd]_{T(\eta_X\otimes\Delta)}
&
& {T(T(X\otimes X)\otimes TX)}
  \ar[rd]^{T(\vartheta_{X\otimes X, X})}
&
\\
& {T(TX\otimes T(X\otimes X))}
  \ar[rd]_{T(\vartheta_{X,X\otimes X})}
&
& TT(X\otimes X\otimes X)
  \ar[d]^{\mu_{X\otimes X\otimes X}}
\\
&
& TT(X\otimes X\otimes X)
  \ar[r]_{\mu_{X\otimes X\otimes X}}
& T(X\otimes X\otimes X)
}\]
in the category $\calC$. Thus, in $\calC$, we have
\begin{align}\label{eq:Kl convol-semigp 4}
  \mu_{X\otimes X\otimes X} \compos T( \vartheta_{X\otimes X,X} \compos(\Delta\otimes\eta_X)) \compos\Delta
= \mu_{X\otimes X\otimes X} \compos T( \vartheta_{X,X\otimes X} \compos(\eta_X\otimes\Delta)) \compos\Delta.
\end{align}
Notice that
\begin{align}\label{eq:Kl convol-semigp 5}
  \vartheta_{Y\otimes Y,Y}
  \compos(\vartheta_{Y,Y}\otimes\id_{TY})
  \compos((f\otimes g)\otimes h)
= \vartheta_{Y,Y\otimes Y}
  \compos(\id_{TY}\otimes\vartheta_{Y,Y})
  \compos(f\otimes(g\otimes h)),
\end{align}
then by \eqref{eq:Kl convol-semigp 2}, \eqref{eq:Kl convol-semigp 3},
\eqref{eq:Kl convol-semigp 4} and \eqref{eq:Kl convol-semigp 5}, we obtain
\begin{align}
(f\ast g)\ast h & =
\mu_Y \compos
  T(\mu_Y \compos T(m)\compos \vartheta_{Y,Y}\compos(\eta_Y\otimes m)) \nonumber \\
& \compos \mu_{Y\otimes Y\otimes Y} \nonumber \\
& \compos T( \vartheta_{Y,Y\otimes Y}
  \compos(\id_{TY}\otimes\vartheta_{Y,Y})
  \compos(f\otimes(g\otimes h)) ) \nonumber \\
& \compos\mu_{X\otimes X\otimes X}\nonumber \\
& \compos T( \vartheta_{X,X\otimes X}
  \compos(\eta_X\otimes\Delta) )
  \compos\Delta \nonumber \\
&= f\ast(g\ast h). \nonumber
\end{align}
Therefore, the convolution is associative, i.e., $\Hom_{\Kl(\scrT)}(X,Y)$ is a semigroup.

Define $e_{\ast} := u \compos_{\Kl} \varepsilon  =  \mu_Y\compos T(u)\compos\varepsilon: ~ X\to TY$.
We next prove that $e_{\ast}$ is the identity of $\Hom_{\Kl(\scrT)}(X,Y)$.
For every $f:X\to TY$, we have
\[ e_{\ast} \ast f = \mu_Y \compos T(m) \compos \mu_{Y\otimes Y} \compos T(e_{\ast}\otimes_{\Kl} f)\compos\Delta \checks{.} \]
{By functoriality of the Kleisli tensor product and associativity of Kleisli composition, we have}
\begin{align}\label{eq:Kl convol-semigp 6}
e_{\ast}\ast f =
  \mu_Y \compos T( \mu_Y\compos T(m)\compos \vartheta_{Y,Y}\compos(u\otimes\eta_Y))
  \compos\mu_{\bfone\otimes Y} \compos
  T(\vartheta_{\bfone,Y} \compos(\varepsilon\otimes f)) \compos \Delta.
\end{align}
Note that we have the following commutative diagram
\[\xymatrix{
  \bfone\otimes Y
  \ar@{~>}[r]^{u\otimes\eta_Y}
  \ar@{~>}[dr]_{\lambda_Y}
& Y\otimes Y \ar[d]^{m}
& Y\otimes \bfone
  \ar@{~>}[l]_{\eta_Y\otimes u}
  \ar@{~>}[dl]^{\rho_Y} \\
  & Y &
}\]
by the monoid object $(Y,m,u)$, it follows that $\lambda_Y = m\compos_{\Kl} (u\otimes_{\Kl}\eta_Y)$ holds, and so,
\begin{align}\label{eq:Kl convol-semigp 7}
  \mu_Y\compos T(m)\compos \vartheta_{Y,Y}\compos(u\otimes\eta_Y) = \eta_Y\compos\lambda_Y.
\end{align}
Dually, we have the following commutative diagram
\[\xymatrix{
  \bfone \otimes X
  \ar@{<~}[r]^{\varepsilon \otimes\eta_X}
  \ar@{<~}[dr]_{\lambda_X^{-1}} & X\otimes X
  \ar@{<~}[d]^{\Delta} & X \otimes \bfone
  \ar@{<~}[l]_{\eta_X\otimes \varepsilon }
  \ar@{<~}[dl]^{\rho_X^{-1}} \\
  & X &
}\]
by the comonoid object $(X,\Delta,\varepsilon)$, it follows that $\lambda_X^{-1} = (\varepsilon \otimes \eta_X)\compos_{\Kl}\Delta$ holds, and so,
\begin{align}\label{eq:Kl convol-semigp 8}
  \mu_{\bfone\otimes Y}\compos T(\vartheta_{\bfone,Y} \compos(\varepsilon\otimes f)) \compos \Delta
= T(\lambda_Y^{-1})\compos f.
\end{align}
By \eqref{eq:Kl convol-semigp 6}, \eqref{eq:Kl convol-semigp 7} and \eqref{eq:Kl convol-semigp 8}, we have
\[ e_{\ast}\ast f
 = \mu_Y \compos T(\eta_Y\compos\lambda_Y)
   \compos T(\lambda_Y^{-1})\compos f = \mu_Y\compos T(\eta_Y)\compos f = f. \]

Similarly, we have
\begin{align*}
  f\ast e_{\ast}
& = \mu_Y \compos T(\mu_Y\compos T(m) \compos \vartheta_{Y,Y}\compos(\eta_Y\otimes u) )
        \compos \mu_{Y\otimes\bfone} \compos T\bigl(\vartheta_{Y,\bfone}
        \compos(f\otimes\varepsilon)) \compos \Delta \\
& = \mu_Y \compos T(\eta_Y\compos\rho_Y) \compos T(\rho_Y^{-1}) \compos f = f.
\end{align*}
Therefore, $\Hom_{\Kl(\scrT)}(X,Y)$ is a monoid under the Kleisli convolution.
\end{proof}


We shall mainly use the case $X=\bfone$. We use the canonical comonoid structure
\[ \Delta_{\bfone}:\bfone\simeq\bfone\otimes\bfone,
 \quad \varepsilon_{\bfone}=\id_{\bfone}. \]
Thus, if $M$ is a monoid object of $\Kl(\scrT)$,
then $\Hom_{\Kl(\scrT)}(\bfone,M)$ is a Kleisli convolution monoid.

\section{Free semimodule monads} \label{sect:free semimod monad}

\textsl{In this section, we recall the free semimodule monad over a commutative semiring. We explain its relation with matrices and realize ordinary matrix multiplication as a Kleisli convolution.}

We first recall the definition of semirings and semimodules, see \cite{Golan} for the basic theory of semirings and semimodules.

\begin{definition}\rm
Recall that a \defines{semiring} is a quintuple $(S,+,0_S,\cdot,1)$ such that $(S,+,0_S)$ is a commutative monoid, $(S,\cdot,1)$ is a monoid, and, for any $r,s,t\in S$ the following formulas
\[ s(t+r)=st+sr, \quad (s+t)r=sr+tr \]
hold. Moreover, the zero element $0_S$ is required to be absorbing, i.e., $0_Ss=0_S=s0_S$ for all $s\in S$.
A semiring $S$ is called \defines{commutative} if its multiplication is commutative.
\end{definition}

\begin{definition}\rm
An \defines{$S$-semimodule} is a commutative monoid $(M,+,0_M)$ together with a scalar multiplication
\[S\times M\longrightarrow M,\quad (s,m)\longmapsto sm\]
such that
\begin{itemize}
  \item $(st)m=s(tm)$;
  \item $(s+t)m=sm+tm$;
  \item $s(m+n)=sm+sn$;
  \item $1m=m$;
  \item $0_Sm=0_M=s0_M$
\end{itemize}
hold for all $s,t\in S$ and $m,n\in M$.
\end{definition}

Let $S$ be a commutative semiring. For a set $X$, let
\begin{align}\label{defeq:semimod}
  \semimod_S(X) = \{a\colon X\to S\mid \supp(a)\text{ is finite}\} \cong \bigoplus_{x\in X} S
\end{align}
Then one can check that it is an $S$-semimodule. In particular, since it is isomorphic to $S^{\oplus X}$, we call it a \defines{free $S$-semimodule}.
Thus, we also write an element of $\semimod_S(X)$ as a formal sum $\sum\limits_{x\in X}a_x[x]$.
If $f\colon X\to Y$ is a map, then, naturally, we have a map $\semimod_S(f): \semimod_S(X)\to \semimod_S(Y)$ which is given by
\[ \semimod_S(f)\bigg(\sum_{x\in X}a_x[x]\bigg)
= \sum_{x\in X}a_x[f(x)]. \]
The unit $\eta^{\semimod_S}_X$ is defined by $\eta^{\semimod_S}_X(x)=[x]$,
and the multiplication $\mu^{\semimod_S}_X\colon\semimod_S^2(X)
:= \semimod_S(\semimod_S(X)) \to \semimod_S(X)$ is defined by
\[ \mu_X^{\semimod_S} \bigg(\sum_{\ell=1}^{n}c_{\ell}
   \bigg[\sum_{x\in X}a_{\ell,x}[x]\bigg]\bigg)
 = \sum_{x\in X}\bigg(\sum_{\ell=1}^{n}c_{\ell}a_{\ell,x}\bigg)[x]. \]

\begin{lemma}[{\!\!\cite[Section 2]{Jacobs2010convexity}}]\label{lemm:semimod-monad}
The triple $\scrM_S = (\semimod_S, \eta^{\semimod_S}, \mu^{\semimod_S})$ is a monad on the set category $\Set$.
\end{lemma}

This lemma is very important for our paper. Thus, we provide a proof of it here.

\begin{proof}
We first show that $\semimod_S$ is an endofunctor on $\Set$.
Let $f\colon X\to Y$ be a map and let $a=\sum\limits_{x\in X}a_x[x]\in\semimod_S(X)$.
For every $y\in Y$, the coefficient of $[y]$ in $\semimod_S(f)(a)$ is $\sum\limits_{\substack{x\in X\\f(x)=y}}a_x$.
This sum is finite since $\supp(a)$ is finite. Moreover, $\supp\big(\semimod_S(f)(a)\big) \subseteq f\big(\supp(a)\big)$,
and hence $\semimod_S(f)(a)$ has finite support. Thus, $\semimod_S(f)$ is well defined.
It is clear that $\semimod_S(\id_X)=\id_{\semimod_S(X)}$.
Let $f\colon X\to Y$ and $g\colon Y\to Z$ be two maps.
For $a=\sum\limits_{x\in X}a_x[x]\in\semimod_S(X)$, we have
\begin{align*}
 \semimod_S(g)\big(\semimod_S(f)(a)\big)
 &=\semimod_S(g)\bigg(\sum_{x\in X}a_x[f(x)]\bigg)\\
 &=\sum_{x\in X}a_x[g(f(x))]\\
 &=\semimod_S(g\compos f)(a),
\end{align*}
i.e., $\semimod_S(g)\compos\semimod_S(f) = \semimod_S(g\compos f)$,
and hence $\semimod_S\colon\Set\to\Set$ is an endofunctor.

We next prove that $\eta^{\semimod_S}\colon\id_{\Set} \to \semimod_S$ is a natural transformation.
Here $[x]$ means the formal sum $1[x]$. For every map $f\colon X\to Y$ and every $x\in X$, we have
\[ \semimod_S(f)(\eta^{\semimod_S}_X(x))
 = \semimod_S(f)([x])
 = [f(x)]
 = \eta^{\semimod_S}_Y(f(x)).\]
Consequently, $\semimod_S(f)\compos\eta^{\semimod_S}_X=\eta^{\semimod_S}_Y\compos f$.

We now prove that $\mu^{\semimod_S}\colon\semimod_S^2 \to \semimod_S$ is a natural transformation.
First, the map $\mu^{\semimod_S}_X$ is well defined. Indeed, if
$\Phi=\sum\limits_{\ell=1}^{n}c_{\ell}
   \Big[\sum\limits_{x\in X}a_{\ell,x}[x] \Big]
   \in\semimod_S^2(X)$, then
\[ \supp\big(\mu^{\semimod_S}_X(\Phi)\big) \subseteq
   \bigcup_{\ell=1}^{n} \supp\bigg(\sum_{x\in X}a_{\ell,x}[x]\bigg),\]
which is a finite set. Next, let $f\colon X\to Y$ be a map. For every $y\in Y$, the coefficient of $[y]$ in $\semimod_S(f)(\mu^{\semimod_S}_X(\Phi))$ is
\begin{align}\label{eq:semimod-monad 1}
    \sum_{\substack{x\in X\\f(x)=y}} \sum_{\ell=1}^{n}c_{\ell}a_{\ell,x}
& = \sum_{\ell=1}^{n}c_{\ell} \bigg(\sum_{\substack{x\in X\\f(x)=y}}a_{\ell,x}\bigg).
\end{align}
On the other hand,
\[ \semimod_S^2(f)(\Phi)
= \sum_{\ell=1}^{n}c_{\ell}
  \bigg[
   \sum_{y\in Y}
   \bigg(
     \sum_{\substack{x\in X\\f(x)=y}}a_{\ell,x}
   \bigg)[y]
  \bigg], \]
and hence the coefficient of $[y]$ in $\mu^{\semimod_S}_Y(\semimod_S^2(f)(\Phi))$ is also \eqref{eq:semimod-monad 1}.
It follows that $\semimod_S(f)\compos\mu^{\semimod_S}_X = \mu^{\semimod_S}_Y\compos\semimod_S^2(f)$. Thus, $\mu^{\semimod_S}$ is natural.

It remains to verify the monad axioms. Let $a=\sum\limits_{x\in X}a_x[x]\in\semimod_S(X)$.
Then we have $\eta^{\semimod_S}_{\semimod_S(X)}(a)=[a] =\bigg[\sum\limits_{x\in X}a_x[x]\bigg]$,
and therefore
\[ \mu^{\semimod_S}_X\big(\eta^{\semimod_S}_{\semimod_S(X)}(a)\big)
 = \mu^{\semimod_S}_X\bigg( \bigg[\sum_{x\in X}a_x[x]\bigg] \bigg)
 = \sum_{x\in X}a_x[x] = a. \]
Hence $\mu^{\semimod_S}_X\compos\eta^{\semimod_S}_{\semimod_S(X)} =\id_{\semimod_S(X)}$.
On the other hand, we have $\semimod_S(\eta^{\semimod_S}_X)(a) = \sum\limits_{x\in X}a_x[[x]]$.
It follows that
\[ \mu^{\semimod_S}_X\big(\semimod_S(\eta^{\semimod_S}_X)(a)\big)
 = \mu^{\semimod_S}_X\left(\sum_{x\in X}a_x[[x]]\right)
 = \sum_{x\in X}a_x[x] = a.\]
Thus, $\mu^{\semimod_S}_X\compos\semimod_S(\eta^{\semimod_S}_X)=\id_{\semimod_S(X)}$.

Finally, we verify the associativity law. Let
\[
 \Omega=
 \sum_{i=1}^{p}d_i
 \bigg[
   \sum_{j=1}^{q_i}c_{ij}
   \bigg[
     \sum_{x\in X}a_{ij,x}[x]
   \bigg]
 \bigg]
 \in\semimod_S^3(X).
\]
Applying $\mu^{\semimod_S}_{\semimod_S(X)}$ first gives
\[ \mu^{\semimod_S}_X(\mu^{\semimod_S}_{\semimod_S(X)}(\Omega))
 = \mu^{\semimod_S}_X
   \bigg(
     \sum_{i=1}^{p}\sum_{j=1}^{q_i}
       d_ic_{ij}
       \bigg[
         \sum_{x\in X}a_{ij,x}[x]
       \bigg]
   \bigg)
 = \sum_{x\in X}
   \bigg(
     \sum_{i=1}^{p}\sum_{j=1}^{q_i}
     d_ic_{ij}a_{ij,x}
   \bigg)[x]. \]
On the other hand, we have
$
 \semimod_S(\mu_X)(\Omega)
= \sum\limits_{i=1}^{p}d_i
 \bigg[
   \sum_{x\in X}
   \bigg(
     \sum\limits_{j=1}^{q_i}c_{ij}a_{ij,x}
   \bigg)[x]
 \bigg],$
and hence
\begin{align*}
    \mu^{\semimod_S}_X\big(\semimod_S(\mu_X)(\Omega)\big)
  = \sum_{x\in X}
    \bigg(
      \sum_{i=1}^{p} d_i
        \bigg( \sum_{j=1}^{q_i}c_{ij}a_{ij,x} \bigg)
    \bigg)[x]
  = \sum_{x\in X}
    \bigg(
      \sum_{i=1}^{p}\sum_{j=1}^{q_i} d_ic_{ij}a_{ij,x}
    \bigg)[x].
\end{align*}
Thus, $\mu^{\semimod_S}_X\compos\mu^{\semimod_S}_{\semimod_S(X)} = \mu^{\semimod_S}_X\compos\semimod_S(\mu^{\semimod_S}_X)$.
Therefore, $\scrM_S=(\semimod_S,\eta^{\semimod_S},\mu^{\semimod_S})$ is a monad on $\Set$.
\end{proof}


\begin{definition} \rm
The construction given in Lemma \ref{lemm:semimod-monad} is usually called the \defines{free $S$-semimodule monad}.
\end{definition}

\begin{remark}\rm
Since $S$ is commutative, the map
\[
 \semimod_S(X)\times\semimod_S(Y)
 \to\semimod_S(X\times Y),\quad
 \bigg(\sum_{x\in X} a_x[x],\sum_{y\in Y} b_y[y]\bigg)
 \mapsto\sum_{(x,y)\in X\times Y} a_xb_y[(x,y)] \]
makes it a commutative monoidal monad.
\end{remark}

By Lemma \ref{lemm:semimod-monad}, we can define a Kleisli category $\Kl(\scrM_S)$ on the set category $\Set$.
This Kleisli category $\Kl(\scrM_S)$ can be regarded as a category of matrices.
To be precise, a Kleisli arrow $f: X \rightsquigarrow Y \in \Hom_{\Kl(\scrM_S)}(X,Y)$ can be seen as a map
\[ f: X\to \semimod_S(Y) = \bigg\{\sum_{y\in Y}a_y[y] ~\bigg|~ a_y\in S\checks{,\ a_y=0\text{ for all but finitely many }y} \bigg\} \]
up to isomorphism. We write $f(x)=\sum\limits_{y\in Y}f_{y,x}[y]$.
For a Kleisli arrow $g: Y \rightsquigarrow Z$, we have
\[ (g\compos_{\Kl} f)_{z,x}=\sum_{y\in Y}g_{z,y}f_{y,x}:
  X \To{f} \semimod_S(Y) \To{\semimod_S(g)} \semimod_S^2(Z) \To{\mu_Z} \semimod_S(Z). \]
Indeed, for every $x\in X$,
\begin{align*}
    (g\compos_{\Kl} f)(x)
& = \mu^{\semimod_S}_Z\semimod_S(g)\bigg(\sum_{y\in Y}f_{y,x}[y]\bigg)\\
& = \mu^{\semimod_S}_Z\bigg(\sum_{y\in Y}f_{y,x}\bigg[\sum_{z\in Z}g_{z,y}[z]\bigg]\bigg)\\
& = \sum_{z\in Z}\bigg(\sum_{y\in Y}g_{z,y}f_{y,x}\bigg)[z].
\end{align*}

Thus, under the above correspondence, Kleisli composition is precisely matrix multiplication,
while Kleisli identity arrows correspond to identity matrices.
We therefore \textbf{denote the Kleisli category $\Kl(\scrM_S)$ by $\Mat_S$}.
Its objects are sets, and a morphism from $X$ to $Y$ is a column-finite $S$-valued matrix
whose rows are indexed by $Y$ and whose columns are indexed by $X$.
We have the following fact.

\begin{lemma}\label{fact:Kl(scrM) is Mat}
The tensor product of $\Mat_S$ is given on objects by $X\otimes X'=X\times X'$.
If $F=(f_{y,x})$ and $H=(h_{y',x'})$ are the matrices associated with Kleisli arrows $f\colon X\to\semimod_S(Y)$
and $h\colon X'\to\semimod_S(Y')$, respectively,
then the matrix associated with $f\otimes h$ is their Kronecker product
\[ F\otimes H  = (f_{y,x}h_{y',x'})_{(y,y')\in Y\times Y',(x,x')\in X\times X'},\]
which is equivalent to
\[ (f\otimes h)(x,x') = \sum_{(y,y')\in Y\times Y'} f_{y,x}h_{y',x'}[(y,y')].\]
\end{lemma}

The following lemma shows that we can construct a monoid object in $\Mat_S$.
In \cite[Section 3.4]{CD2017}, Contreras and Duman introduced this construction.
But the objects considered here are somewhat distinct.
Hence, for the reader's ease, we nevertheless include a complete proof of the lemma.

\begin{lemma}\label{lemm:monoid object of MatS}
Let $X$ be a finite set and put $E_X=X\times X$. Define
\[ m_X\colon E_X\times E_X\to\semimod_S(E_X), \quad
 m_X((i,j),(p,q))=
 \begin{cases}
 [(i,q)],&j=p,\\
 0,&j\ne p,
 \end{cases}
\]
and, for an arbitrary singleton set $\bfone=\{\bullet\}$, define
\[ u_X: \bfone \to \semimod_S(E_X) \cong S^{\oplus E_X},\quad u_X(\bullet)=\sum_{i\in X}[(i,i)].\]
Then $(E_X,m_X,u_X)$ is a monoid object of $\Mat_S$.
\end{lemma}

\begin{proof}
\checks{In the following calculations, we use the $S$-bilinear extension of $m_X$ and keep the same notation. We identify each $(i,j)$ with its basis element $[(i,j)]$.}
To see this directly, denote $(i,j)$ by $e_{ij}$. The definition gives $m_X(e_{ij},e_{pq})=\delta_{jp}e_{iq}$,
where a zero coefficient means the zero formal sum.
Hence, we have
\begin{align*}
 m_X(m_X(e_{ij},e_{pq}),e_{rs}) = \delta_{jp}\delta_{qr}e_{is}, \quad
 m_X(e_{ij},m_X(e_{pq},e_{rs})) = \delta_{qr}\delta_{jp}e_{is}.
\end{align*}
The two expressions are equal. By $S$-bilinearity, the associativity diagram
\[\xymatrix@C=1.5cm{
  E_X \times E_X \times E_X
  \ar[r]^{\checks{m_X \otimes_{\Kl} \eta_{E_X}^{\semimod_S}}}
  \ar[d]_{\checks{\eta_{E_X}^{\semimod_S} \otimes_{\Kl} m_X}}
& E_X \times E_X
  \ar[d]^{m_X}
  \\
  E_X \times E_X
  \ar[r]_{m_X}
& E_X }\]
commutes \checks{in $\Mat_S$}. Moreover,
\[
 m_X\bigg(\sum_{p\in X}e_{pp},e_{ij}\bigg)=e_{ij}
 \quad\text{and}\quad
 m_X\bigg(e_{ij},\sum_{p\in X}e_{pp}\bigg)=e_{ij}.
\]
Thus, the two unit diagrams also commute.
\end{proof}

\begin{remark}\rm
Under the canonical notations, we have $\semimod_S(E_X) = \semimod_S(X\times X) \cong \Mat_X(S)$,
the basis element $[(i,j)]$ corresponds to the matrix unit $e_{ij}$.
The multiplication $m_X$ is therefore determined by $e_{ij}e_{pq}=\delta_{jp}e_{iq}$,
and $u_X(*)$ corresponds to the identity matrix $I_X=\sum\limits_{i\in X}e_{ii}$.
Thus, the above monoid object is precisely the full matrix semiring over $S$, expressed in the Kleisli category $\Mat_S$.
\end{remark}

An earlier version of the following proposition was formulated over the complex field $\mathbb{C}$, see for example \cite{Kha2013}.
In \cite[Example 5]{Saigo2021}, it is shown by Saigo that every finite pair groupoid induces a convolution algebra, which must be isomorphic to a full matrix algebra.
Saigo's example generalizes this earlier version to the more general setting of rigs.
The following proposition expresses this standard fact in terms of the Kleisli category $\Mat_S$.

\begin{proposition}\label{prop:matrix-convolution}
For every finite set $X$, there is an isomorphism of $S$-algebras
\[ \Phi_X: \Hom_{\Mat_S}(\bfone,E_X) \to \Mat_X(S). \]
Under this isomorphism, Kleisli convolution is ordinary matrix multiplication.
\end{proposition}

\begin{proof}
A Kleisli arrow $\alpha: \bfone=\{\bullet\} \rightsquigarrow E_X = X\times X$
in $\Hom_{\Mat_S}(\bfone,E_X)$ is uniquely determined by
$\alpha(\bullet)=\sum\limits_{(i,j)}a_{ij}[(i,j)]$.
Write $\Phi_X(\alpha)$ as $(a_{ij})$. This is clearly an isomorphism of $S$-semimodules.
Indeed, its inverse sends a matrix $(a_{ij}) \in \Mat_X(S)$ to the correspondence $\bullet \mapsto$ $\sum\limits_{(i,j)}a_{ij}[(i,j)]$.
Now, let
\[ \alpha(\bullet)=\sum_{(i,j)\in E_X}a_{ij}[e_{ij}]
   \quad\text{and}\quad
   \beta(\bullet)=\sum_{(p,q)\in E_X}b_{pq}[e_{pq}]. \]
The tensor product of these two Kleisli arrows sends the element $\bullet \in \bfone$ to the finite sum
$\sum\limits_{(i,j),(p,q) \in E_X} a_{ij}b_{pq}[(e_{ij},e_{pq})]$.
Applying $m_X$ to it, we obtain that this application deletes all summands for which $j\ne p$,
and then sends $(e_{ij},e_{jq})$ to $e_{iq}$, see Lemma \ref{lemm:monoid object of MatS}. Consequently,
\[ (\alpha*\beta)(\bullet) = \sum_{(i,q)\in E_X}\bigg(\sum_{j\in X}a_{ij}b_{jq}\bigg)[e_{iq}]. \]
Then the coefficient of $(i,q)$ in $\alpha*\beta$ is $\sum\limits_{j\in X}a_{ij}b_{jq}$.
This is the $(i,q)$-entry of the ordinary matrix product $(a_{ij})(b_{ij})$. Hence
\[ \Phi_X(\alpha*\beta)=\Phi_X(\alpha)\Phi_X(\beta). \]
Finally, $u_X(*)=\sum\limits_{i\in X}[e_{ii}]$ is sent to the identity matrix.
Therefore, $\Phi_X$ is an isomorphism of unital $S$-algebras.
\end{proof}

\begin{remark}\label{rem:basis} \rm
The finite set $X$ should be regarded as a choice of basis. Therefore, an
object of the form $\semimod_S(X)$ is a free $S$-semimodule with a specified
basis. Changing the basis gives an isomorphic representation, but it need not
give the same Kleisli arrow.
\end{remark}

\section{Groups and their Kleisli convolution representations} \label{sect:group KCR}

\textsl{In this section, we introduce finite and linear Kleisli convolution representations of groups. We compare their categories with the categories of finite permutation representations and finite-dimensional linear representations, respectively.}

\subsection{Representations of groups}

We recall some basic notions of permutation representations of groups, see \cite{BiggsWhite1979}.
Let $G$ be a group. Recall that a \defines{finite permutation representation} of $G$ is a pair $(X,\sigma_X)$ consisting of a finite set $X$ and a group homomorphism
\[ \sigma_X\colon G\to\operatorname{Sym}(X) = \{\varsigma:X\to X\mid\varsigma\text{ is a bijection}\}, \]
where $\operatorname{Sym}(X)$ is a group under the composition of maps.
Equivalently, a finite permutation representation of $G$ is a \defines{finite left $G$-set},
i.e., a finite set $X$ with a left $G$-action $G\times X\to X$, $g\cdot x=\sigma_X(g)(x)$.

Let $(X,\sigma_X)$ and $(Y,\sigma_Y)$ be two finite permutation representations of $G$.
A \defines{$G$-homomorphism of permutation representations} from $(X,\sigma_X)$ to $(Y,\sigma_Y)$ is a map $f:X\to Y$ such that $f\compos\sigma_X(g)=\sigma_Y(g)\compos f$ for every $g\in G$.
Equivalently, $f$ satisfies $f(g\cdot x)=g\cdot f(x)$ for all $g\in G$ and $x\in X$.
The composition of two $G$-homomorphisms $f_1:(X,\sigma_X)\to(Y,\sigma_Y)$ and $f_2:(Y,\sigma_Y)\to(Z,\sigma_Z)$ is defined to be their usual composition $f_2\compos f_1:X\to Z$.
In this case, for every $g\in G$, we have
\begin{align*}
 (f_2\compos f_1)\compos\sigma_X(g)
 = f_2\compos(f_1\compos\sigma_X(g)) = f_2\compos(\sigma_Y(g)\compos f_1)
 = (f_2\compos\sigma_Y(g))\compos f_1 = \sigma_Z(g)\compos(f_2\compos f_1).
\end{align*}
Hence, $f_2\compos f_1$ is also a $G$-homomorphism.
For every $(X,\sigma_X)$, the identity map $\id_X$ is a homomorphism from $(X,\sigma_X)$ to itself.
The associativity of composition and the identity laws follow from those for maps of sets.
Therefore, finite permutation representations of $G$ and their homomorphisms form a category, which is denoted by $\Perm(G)$.

Now, we recall some basic notions of linear representations, see \cite[Chapters 1, 4, 6]{EGHLSVY2011}.
Let $\kk$ be a field. Recall that a \defines{finite-dimensional linear representation} of $G$ over $\kk$ is
a pair $(V,\rho_V)$ consisting of a finite-dimensional $\kk$-vector space $V$ and a group homomorphism
\[ \rho_V:G\to\operatorname{GL}(V),\]
where $\operatorname{GL}(V)$ is the group of all invertible $\kk$-linear maps from $V$ to itself.
Equivalently, a finite-dimensional linear representation of $G$ over $\kk$ is a finite-dimensional left module over the group algebra $\kk G$, where the left $\kk G$-action is given by $\kk G\times V \to V$, $g\cdot v=\rho_V(g)(v)$.

Let $(V,\rho_V)$ and $(W,\rho_W)$ be two finite-dimensional linear representations of $G$ over $\kk$.
A \defines{$\kk G$-homomorphism of linear representations} from $(V,\rho_V)$ to $(W,\rho_W)$ is
a $\kk$-linear map $\varphi:V\to W$ such that $\varphi\compos\rho_V(g)=\rho_W(g)\compos\varphi$ for every $g\in G$.
Equivalently, $\varphi$ satisfies $\varphi(g\cdot v)=g\cdot\varphi(v)$ for all $g\in G$ and $v\in V$.
Thus, the $\kk G$-homomorphisms are precisely the homomorphisms of left $\kk G$-modules.

The composition of two $\kk G$-homomorphisms $\varphi:(V,\rho_V)\to(W,\rho_W)$ and $\psi:(W,\rho_W)\to(U,\rho_U)$ is defined to be their usual composition $\psi\compos\varphi:V\to U$.
In this case, for every $g\in G$, we have
\begin{align*}
   (\psi\compos\varphi)\compos\rho_V(g)
 = \psi\compos(\varphi\compos\rho_V(g))
 = \psi\compos(\rho_W(g)\compos\varphi)
 = (\psi\compos\rho_W(g))\compos\varphi
 = \rho_U(g)\compos(\psi\compos\varphi).
\end{align*}
Hence, $\psi\compos\varphi$ is also a $\kk G$-homomorphism. For every $(V,\rho_V)$, the identity map $\id_V$ is a homomorphism from $(V,\rho_V)$ to itself.
The associativity of composition and the identity laws follow from those for linear maps.
Therefore, finite-dimensional linear representations of $G$ over $\kk$ and their $\kk G$-homomorphisms form a category, which is denoted by $\rep_{\kk}(G)$.

It should be noted that the following fact holds.

\begin{fact}\label{fact:Perm vs rep}
Categories $\Perm(G)$ and $\rep_{\kk}(G)$ are different.
\end{fact}

\noindent Indeed, there is a linearization functor $\kk[-]\colon\Perm(G)\to\rep_{\kk}(G)$,
but its essential image consists of permutation modules and is usually a proper subcategory of $\rep_{\kk}(G)$.

\subsection{Kleisli convolution representations of groups}

For a set $X$, let $\calP(X):=\{A: A\subseteq X\}$.
Let \[ \scrP
= (\calP,\, \eta^{\calP}:\id_{\Set}\to\calP, \, \mu^{\calP}:\calP^2\to \calP) \]
be the triple defined by:
\begin{enumerate}[label=($\mathscr{P}$\arabic*)]
  \item for any set $X$ in $\Set$ and any $x\in X$, $\eta_X(x)=\{x\}$ is a set containing only one element $x$;
    \label{P1}
  \item for any $\mathcal A \in \calP(\calP(X))$, $\mu_X(\mathcal A)$ is the union $\bigcup\limits_{A\in\mathcal A}A$ $(\in \calP(X))$.
    \label{P2}
\end{enumerate}
Here, for a map $f:X\to Y$, we define $\calP(f)$ to be the map
\[ \calP(f): \calP(X) \to \calP(Y) \]
sending each $A$ to the \defines{direct image} $f(A):=\{f(a)\in Y : a\in A\} \in \calP(Y)$.
Then, by \cite[Lemma 2.2]{WHL2026}, $\scrP$ is a monad on $\Set$, and we call it a \defines{powerset monad}.
In this case, we obtain a Kleisli category $\Kl(\scrP)$ whose Kleisli arrow $X \rightsquigarrow Y$ is a map $X \to \calP(Y)$.

For a finite set $X$, the invertible Kleisli arrows $X \rightsquigarrow X$ are precisely the graphs of permutations of $X$,
i.e., \[ \mathrm{Iso}_{\Kl(\scrP)}(X,X) \cong \mathrm{Sym}(X). \]
Indeed, a Kleisli arrow $f:X \rightsquigarrow X$ is invertible if and only if there is a Kleisli arrow $g:X \rightsquigarrow X$ such that
\begin{align}\label{eq:eta 08302325}
  \eta^{\calP}_X = g \compos_{\Kl} f:
& X \To{f} \calP(X) \To{\calP(g)} \calP(\calP(X)) \To{\mu^{\calP}_X} \calP(X) \\
& x\mapsto \bigcup_{y\in f(x)}g(y) = \eta^{\calP}_X(x) = \{x\} \nonumber
\end{align}
and
\begin{align}\label{eq:eta 08302326}
  \eta^{\calP}_X = f \circ_{\Kl} g:
& X \To{g} \calP(X) \To{\calP(f)} \calP(\calP(X)) \To{\mu^{\calP}_X} \calP(X) \\
& x\mapsto \bigcup_{y\in g(x)}f(y) = \eta^{\calP}_X(x) = \{x\} \nonumber
\end{align}
hold. For every $x\in X$, it follows from \eqref{eq:eta 08302325} that there is some $y\in f(x)$ such that $x\in g(y)$. Applying \eqref{eq:eta 08302326} to $y$, we obtain $\bigcup\limits_{z\in g(y)}f(z)=\{y\}$. Since $x\in g(y)$, it follows that $f(x)\subseteq\{y\}$. On the other hand, \eqref{eq:eta 08302325} implies that $f(x)$ is non-empty. Hence, $f(x)=\{y\}$. Therefore, for every $x\in X$, there is a unique element $\bar f(x)\in X$ such that $f(x)=\{\bar f(x)\}$. Similarly, for every $x\in X$, there is a unique element $\bar g(x)\in X$ such that $g(x)=\{\bar g(x)\}$. Equations \eqref{eq:eta 08302325} and \eqref{eq:eta 08302326} imply that $\bar g\compos\bar f=\id_X=\bar f\compos\bar g$. Thus, $\bar f:X\to X$ induced by the invertible Kleisli arrow $f:X\rightsquigarrow X$ is a bijection. Conversely, every permutation $\sigma:X\to X$ induces an invertible Kleisli arrow $x\mapsto\{\sigma(x)\}$.

As in Section \ref{sect:free semimod monad}, define $E_X=X\times X$.
Then, parallel to Lemma \ref{lemm:monoid object of MatS}, the Kleisli category $\Kl(\scrP)$ has a monoid object
\begin{center}
$(E_X, m_X: E_X\times E_X\to \calP(E_X), u_X: \bfone\to \calP(E_X))$.
\end{center}
In this case, the tensor product of objects is given by the Cartesian product, i.e., $\otimes = \times$,
and, by Lemma \ref{lemm:Klcorresp}, we have a bijection
\begin{align}\label{iso:Kleisli P}
  \Klcorresp_{\scrP,X}: \KlHom_{\scrP}(X):=\Hom_{\Kl(\scrP)}(\bfone,E_X) \to \calP(E_X),
  \quad (f:\bfone \rightsquigarrow E_X) \mapsto f(\bullet)
\end{align}
Thus, each element of $\KlHom_{\scrP}(X)$ can be seen as a subset of $E_X=X\times X = X\otimes X$.

We can define a multiplication in $\KlHom_{\scrP}(X)$ by the convolution
\begin{align*}
f\ast g: \ \bfone & \To{\Delta} \calP(\bfone\times \bfone) \\
& \To{\calP(f\times g)} \calP(\calP(E_X)\times \calP(E_X))
  \To{\calP(\vartheta_{E_X,E_X})} \calP^2(E_X\times E_X)
  \To{\mu_{E_X\times E_X}} \calP(E_X\times E_X) \\
& \To{\calP(m_X)} \calP^2(E_X)
  \To{\mu_{E_X}} \calP(E_X),
\end{align*}
which is called a \defines{Kleisli convolution}, where $\vartheta_{E_X,E_X}$ is called a \defines{Fubini map} in this case.
The following lemma shows that $\KlHom_{\scrP}(X)$ is a monoid.

\begin{lemma} \label{lemm:KlHom-monoid}
The Kleisli Hom-space $\KlHom_{\scrP}(X)$ is a monoid whose multiplication is given by the Kleisli convolution and whose unit is $u_X:\bfone \rightsquigarrow E_X$.
Here, $u_X$ satisfies $u_X(\bullet) = \{(x,x):x\in X\}$.
\end{lemma}

\begin{proof}
Let $f:\bfone\rightsquigarrow E_X$ be an element of $\KlHom_{\scrP}(X)$.
By Definition \ref{def:Kleisli cat}, $f$ is a map $\bfone \to \calP(E_X)$ in $\Set$.
Since $\bfone=\{\bullet\}$ is a singleton set,
the map $f$ is uniquely determined by the subset $f(\bullet)$ of $E_X=X\times X$.

For each $x\in X$, consider all the elements of $f(\bullet)$
($\subseteq E_X = X\times X$) whose second coordinate is $x$.
Their first coordinates form the subset $\{y\in X:(y,x)\in f(\bullet)\}$ of $X$.
Therefore, we can define a map $\widehat{f}:X\to\calP(X)$ by $\widehat{f}(x)=\{y\in X:(y,x)\in f(\bullet)\}$.
Clearly, $\widehat{f}$ is a Kleisli arrow $X\rightsquigarrow X$,
and then we obtain a correspondence
\begin{center}
  $\KlHom_{\scrP}(X) \to \Hom_{\Kl(\scrP)}(X,X)$, $f\mapsto \widehat{f}$,
\end{center}
which is a bijection. Indeed:
\begin{itemize}
  \item  Let $f,g\in\KlHom_{\scrP}(X)$ and suppose that $\widehat{f}=\widehat{g}$. For any $(y,x)\in X\times X$, we have
    \[ (y,x)\in f(\bullet) ~\Longleftrightarrow~ y\in\widehat{f}(x)
    ~\Longleftrightarrow~ y\in\widehat{g}(x) ~\Longleftrightarrow~ (y,x)\in g(\bullet). \]
    It follows that $f(\bullet)=g(\bullet)$. Since $\bfone$ is a singleton set, we have $f=g$, and so $f\mapsto\widehat{f}$ is injective.
  \item On the other hand, let $a:X\rightsquigarrow X$ be an arbitrary Kleisli arrow in $\Hom_{\Kl(\scrP)}(X,X)$.
    Define $f_a:\bfone\to\calP(E_X)$ by $f_a(\bullet) = \{(y,x)\in X\times X:y\in a(x)\}$. Then $f_a$ is a Kleisli arrow $\bfone \rightsquigarrow E_X$ in $\KlHom_{\scrP}(X)$.
    For every $x\in X$, we have
        \[ \widehat{f_a}(x) = \{y\in X:(y,x)\in f_a(\bullet)\} = \{y\in X:y\in a(x)\} = a(x).\]
Therefore, $\widehat{f_a}=a$. Hence, $f\mapsto\widehat{f}$ is surjective.
\end{itemize}

Next we show that Kleisli convolution is associative.
For any $f,g\in\KlHom_{\scrP}(X)$, by the definitions of the Kleisli convolution and $m_X$,
for every $x\in X$ we have
\begin{align}\label{eq:KlHom-monoid}
 \widehat{f\ast g}(x)
 &=\{z\in X:(z,x)\in(f\ast g)(\bullet)\} \nonumber \\
 &=\{z\in X:\text{there is some }y\in X\text{ such that }
       (z,y)\in f(\bullet)\text{ and }(y,x)\in g(\bullet)\} \nonumber \\
 &=\bigcup_{y\in\widehat{g}(x)}\widehat{f}(y) \nonumber \\
 &=(\widehat{f}\compos_{\Kl}\widehat{g})(x).
\end{align}
Hence, $\widehat{f\ast g} = \widehat{f}\compos_{\Kl}\widehat{g}$.
Then for any $f,g,h\in\KlHom_{\scrP}(X)$, we have
\begin{align*}
  \widehat{(f\ast g)\ast h} = (\widehat{f}\compos_{\Kl}\widehat{g}) \compos_{\Kl}\widehat{h}
= \widehat{f}\compos_{\Kl} (\widehat{g}\compos_{\Kl}\widehat{h})
= \widehat{f\ast(g\ast h)}.
\end{align*}
Since the correspondence $f\mapsto\widehat{f}$ is a bijection,
it follows that $(f\ast g)\ast h=f\ast(g\ast h)$.

Moreover, for $u_X(\bullet)=\{(x,x):x\in X\}$, we have $\widehat{u_X}(x)=\{x\}=\eta_X^{\calP}(x)$,
i.e., $\widehat{u_X}=\eta_X^{\calP}$. Furthermore,
\[ \widehat{u_X\ast f} = \eta_X^{\calP}\compos_{\Kl}\widehat{f}
= \widehat{f} = \widehat{f}\compos_{\Kl}\eta_X^{\calP} = \widehat{f\ast u_X}.\]
The bijectivity of $f\mapsto\widehat{f}$ gives
$u_X\ast f=f=f\ast u_X$. Therefore, $\KlHom_{\scrP}(X)$ is a monoid whose multiplication is the Kleisli convolution and whose unit is $u_X$.
\end{proof}

In the proof of Lemma \ref{lemm:KlHom-monoid}, we show that the convolution has an explicit form
since a Kleisli arrow $f:\bfone=\{\bullet\} \to \calP(E_X)$ is uniquely determined by the subset $f(\bullet)$ of $E_X$.
To be precise, if $f(\bullet)=A$ and $g(\bullet)=B$, then the equation \eqref{eq:KlHom-monoid} shows that
\[ (f\ast g)(\bullet)
 = \{(i,q)\in X\times X\mid
   (i,j)\in A\text{ and }(j,q)\in B
   \text{ for some }j\in X\}.\]
The bijection \eqref{iso:Kleisli P} implies that $\calP(E_X)$ can be seen as a monoid whose multiplication is induced by using the Kleisli convolution in $\Hom_{\Kl(\scrP)}(\bfone,E_X)$,
i.e., for any two subsets $A, B \subseteq E_X$, we can find two Kleisli arrows $f, g\in \KlHom_{\scrP}(X)$ with $f(\bullet)=A$ and $g(\bullet)=B$, and define $A\ast B = \Klcorresp_{\scrP,X}(f \ast g)$ by using \eqref{iso:Kleisli P}.
Thus, we have
\begin{align}\label{eq:sets Kl convol}
A\ast B &= \Klcorresp_{\scrP,X}(f \ast g) = (f\ast g)(\bullet) 
\end{align}

\begin{lemma}\label{lemm:units-power-convolution}
Keep the notation from \ref{lemm:KlHom-monoid}. The invertible elements of $\KlHom_{\scrP}(X)$ are precisely the Kleisli arrows $f:\bfone \rightsquigarrow E_X$ of the form $f(\bullet)=\{(\sigma(x),x)\mid x\in X\}$,
where $\sigma$ is a permutation of $X$. Moreover, the permutation $\sigma$ is uniquely determined by $f$.
\end{lemma}

\begin{proof}
Suppose that the subsets $A,B\subseteq X\times X$ determine two mutually inverse elements of $\KlHom_{\scrP}(X)$.
Then we have
\begin{align}\label{eq:units-power-convolution}
  B\ast A = A\ast B = \Klcorresp_{\scrP,X}(u_X) = u_X(\bullet) = \{(x,x)\mid x\in X\}
\end{align}
by Lemma \ref{lemm:KlHom-monoid}.
For every $x\in X$, by \eqref{eq:sets Kl convol}, $(x,x)\in A\ast B$ shows that
there is some $y\in X$ such that $(x,y)\in A$ and $(y,x)\in B$.
This element $y$ is unique. Indeed, if $(x,y')\in A$, then $(y,x)\in B$ and $(x,y')\in A$ imply that $(y,y')\in B\ast A$.
Since $B\ast A = u_X(\bullet)$, we obtain $y=y'$.
Hence, for every $x\in X$, there is a unique element $\underline{x}\in X$ such that $(x,\underline{x})\in A$.
In the same way, for every $x\in X$, there is a unique element $\overline{x}\in X$ such that $(x,\overline{x})\in B$.
By \eqref{eq:units-power-convolution}, we have
\[ \overline{~\underline{x}~}=x \quad \text{and} \quad \underline{~\overline{x}~}=x.\]
Thus, $\underline{?}: x\mapsto \underline{x}$ and $\overline{?}: x\mapsto \overline{x}$ are bijections defined on $X$.
Taking $\sigma=\overline{?}={\underline{?}}^{-1}$ and replacing $x$ by $\overline{x}$, we obtain
\[A=\{(\sigma(x),x)\mid x\in X\}.\]
The permutation $\sigma$ is clearly uniquely determined by $A$.

Conversely, let $\sigma$ be a permutation of $X$, and let $A=\{(\sigma(x),x)\mid x\in X\}$ and $B=\{(\sigma^{-1}(x),x)\mid x\in X\}$. It follows directly from the definition of Kleisli convolution that
\[ A\ast B=\{(x,x)\mid x\in X\} \quad\text{and}\quad B\ast A=\{(x,x)\mid x\in X\}.\]
Therefore, the Kleisli arrow determined by $A$ is invertible.
\end{proof}

Lemmas \ref{lemm:KlHom-monoid} and \ref{lemm:units-power-convolution} show that $\KlHom_{\scrP}(X)$ is a monoid
and provide a description of invertible Kleisli arrows in $\KlHom_{\scrP}(X)$.
Now, we can introduce a new representation for each group $G$ by using the Kleisli category $\Kl(\scrP)$.

\begin{definition}\rm \label{def:Kl convol rep gp} \
\begin{itemize}
  \item[(1)] Let $G$ be a group. A \defines{{\rm(}finite{\rm)} Kleisli convolution representation} of $G$ on a finite set $X$ is a monoid homomorphism
    \[ \rho:G \to \KlHom_{\scrP}(X) = \Hom_{\Kl(\scrP)}(\bfone,E_X).\]
    Equivalently, it is a pair $(X,\rho)$ given by a finite set $X$ and a monoid homomorphism $\rho:G \to \KlHom_{\scrP}(X)$.
  \item[(2)] Since every $g\in G$ is invertible, Lemma \ref{lemm:units-power-convolution} shows that there is a \textbf{unique} permutation $\sigma_{X,g}$ of $X$ such that
    \[\rho(g)(\bullet) = \{(\sigma_{X,g}(x),x)\mid x\in X\}.\]
    Let $(X,\rho_X)$ and $(Y,\rho_Y)$ be two Kleisli convolution representations of $G$.
    A \defines{homomorphism} from $(X,\rho_X)$ to $(Y,\rho_Y)$ is a map $f:X\to Y$ satisfying
    \[ f(\sigma_{X,g}(x)) = \sigma_{Y,g}(f(x)) \text{ for all } g\in G \text{ and } x\in X.\]
    In other words, for every $g\in G$, the following diagram is commutative:
    \[
     \xymatrix{
     X \ar[r]^{\sigma_{X,g}}\ar[d]_{f}
       & X\ar[d]^{f} \\
     Y \ar[r]_{\sigma_{Y,g}}
       & Y.
     } \]
\end{itemize}
\end{definition}

\begin{theorem}\label{thm:KLcrep gp}
Let $G$ be a group. All finite Kleisli convolution representations of $G$ and all homomorphisms between two Kleisli convolution representations of $G$ form a category $\KCrep(G)$.
\end{theorem}

\begin{proof}
%
Let $(X,\rho_X)$, $(Y,\rho_Y)$ and $(Z,\rho_Z)$ be three finite Kleisli convolution representations of $G$,
and let $f_1:(X,\rho_X)\to(Y,\rho_Y)$ and $f_2:(Y,\rho_Y)\to(Z,\rho_Z)$ be two homomorphisms.
Then they are maps $f_1:X\to Y$ and $f_2:Y\to Z$ satisfying $f_1\compos\sigma_{X,g}=\sigma_{Y,g}\compos f_1$
and $f_2\compos\sigma_{Y,g}=\sigma_{Z,g}\compos f_2$, respectively.
Thus, we have the following diagram
\[\xymatrix{
 X \ar[r]^{\sigma_{X,g}}\ar[d]_{f_1} \ar@/_2pc/[dd]_{f_2\compos f_1}
   & X\ar[d]^{f_1} \ar@/^2pc/[dd]^{f_2\compos f_1} \\
 Y \ar[r]^{\sigma_{Y,g}}\ar[d]_{f_2}
   & Y\ar[d]^{f_2} \\
 Z \ar[r]^{\sigma_{Z,g}}
   & Z.
} \]
commutes. It follows that $(f_2\compos f_1)\compos \sigma_{X,g} = \sigma_{Z,g}\compos(f_2\compos f_1)$,
i.e., $f_2\compos f_1$ is a homomorphism from $(X,\rho_X)$ to $(Z,\rho_Z)$.
The composition of homomorphisms is well-defined.

For every object $(X,\rho_X)$, the identity map $\id_X:X\to X$
satisfies $\id_X\compos\sigma_{X,g} = \sigma_{X,g} = \sigma_{X,g}\compos\id_X$ for every $g\in G$.
Then $\id_X$ is a homomorphism from $(X,\rho_X)$ to itself, and it is the identity homomorphism of $(X,\rho_X)$.

Finally, the composition of homomorphisms is associative because it is the usual composition of maps.
Moreover, it is clear to see that $\id_Y\compos f=f=f\compos\id_X$ for every homomorphism $f:(X,\rho_X)\to(Y,\rho_Y)$.
Therefore, the objects, homomorphisms, compositions and identity homomorphisms defined above satisfy all the axioms of a category. Hence, they form the category $\KCrep(G)$.
\end{proof}

Theorem \ref{thm:KLcrep gp} is undoubtedly important in this paper. This theorem shows that we can define the category of Kleisli convolution representations of groups.

\begin{definition}\rm
A \defines{Kleisli convolution representation category} $\KCrep(G)$ of a group $G$ is a category whose objects are Kleisli convolution representations of $G$ and whose morphisms are homomorphisms between two Kleisli convolution representations of $G$.
\end{definition}

\begin{theorem}\label{thm:KCrep-perm rep}
There is an equivalence of categories $\KCrep(G) \simeq \Perm(G)$.
\end{theorem}

\begin{proof}
Let $(X,\rho)$ be a Kleisli convolution representation of $G$. Then we have a monoid homomorphism
\[ \rho: G \to \KlHom_{\scrP}(X), ~ g\mapsto \big(\rho(g): \bfone \to \calP(E_X), \bullet\mapsto \{(\sigma_{X,g}(x),x):x\in X\} \big).\]
For every $g\in G$, let $\sigma_{X,g}$ be the unique permutation
of $X$ determined by $\rho(g)(\bullet) = \{(\sigma_{X,g}(x),x)\mid x\in X\}$.

Notice that for any two permutations $\sigma$ and $\tau$ of $X$, we have
$\{(\sigma(x),x): x\in X\} \ast \{(\tau(x),x): x\in X\} = \{(\sigma(\tau(x)),x): x\in X\}$
by Kleisli convolution. Thus, for any two elements $g,h\in G$, we have
\begin{align*}
  \rho(g)(\bullet) \ast \rho(h)(\bullet)
& = \{((\sigma_{X,g}\compos\sigma_{X,h})(x), x) : x\in X\} \\
& \mathop{=}\limits^{\spadesuit}~ \{(\sigma_{X,gh}(x), x) : x\in X\}\\
& = \rho(gh)(\bullet),
\end{align*}
where $\spadesuit$ holds since $\rho$ is a monoid homomorphism. By the uniqueness in Lemma \ref{lemm:units-power-convolution}, we have $\sigma_{X,gh}=\sigma_{X,g}\compos\sigma_{X,h}$ and $\sigma_{X,e}=\id_X$. Thus, $G\to\mathrm{Sym}(X)$, $g\mapsto\sigma_{X,g}$, is a group homomorphism and gives a left $G$-action $G\times X\to X$, $g\cdot x=\sigma_{X,g}(x)$.
Then we obtain a correspondence
\[ (X, \rho) \mapsto (X, \sigma_X) \]
from the object $(X,\rho)$ in $\KCrep(G)$ to the object $(X,\sigma_X)$ in $\Perm(G)$.

For any homomorphism $f:(X,\rho_X)\to(Y,\rho_Y)$ of Kleisli convolution representations,
we have  $f(g\cdot x)=g\cdot f(x)$ for all $g\in G$ and $x\in X$.
Then, naturally, $f$ is also a morphism of permutation representations,
and we therefore obtain a correspondence
\begin{align} \label{eq:KCrep-perm rep-morphism}
  \Hom_{\KCrep(G)}((X,\rho_X), (Y,\rho_Y)) & \to \Hom_{\Perm(G)}((X,\sigma_X), (Y,\sigma_Y)) \\
  f & \mapsto f \nonumber
\end{align}
Thus, we obtain a functor
\[F: \KCrep(G) \to \Perm(G)\]
sending each Kleisli convolution representation $(X,\rho_X)$ of $G$ to a permutation representation $(X,\sigma_X)$ of $G$.

Conversely, let $(X, \sigma_X: g\mapsto (x\mapsto g\cdot x))$ be an arbitrary finite permutation representation of $G$.
For every $g\in G$, define $\rho_X(g):\bfone \to \calP(E_X)$ by $\rho_X(g)(\bullet) = \{(g\cdot x,x)\mid x\in X\}$.
For $g,h\in G$, we have
\[(\rho_X(g)\ast\rho_X(h))(\bullet)= \{(g\cdot(h\cdot x),x)\mid x\in X\}
 = \{((gh)\cdot x,x)\mid x\in X\}
 = \rho_X(gh)(\bullet).\]
Moreover, $\rho_X(e)(\bullet) = \{(x,x)\mid x\in X\} = u_X(\bullet)$.
Then $\rho_X:G\to\KlHom_{\scrP}(X)$ is a monoid homomorphism.
Therefore, any permutation representation $(X,\sigma_X)$ induces a Kleisli convolution representation $(X,\rho_X)$.
Notice that \eqref{eq:KCrep-perm rep-morphism} is a bijection, whose inverse induces a correspondence from the morphisms in $\Perm(G)$ to the morphisms in $\KCrep(G)$.
Then we obtain a functor
\[ G: \Perm(G)\longrightarrow\KCrep(G).\]
One can check that the two constructions are inverse to each other on both objects and morphisms.
Thus, we have an equivalence $\KCrep(G)\cong\Perm(G)$ of categories.
\end{proof}

\subsection{Linear Kleisli convolution representations} \label{subsect:linear Kl convol rep}

Let $\kk$ be a field. Recall that \eqref{defeq:semimod} shows that
\[ \semimod_{\kk}(X)
 = \bigg\{ \sum_{x\in X}a_x[x]\ :\
     a_x\in\kk,\ a_x=0\text{ for all but finitely many }x
   \bigg\} \cong \bigoplus_{x\in X}\kk\]
holds for every set $X$; and for each map $f:X\to Y$, applying $\semimod_{\kk}$ to $f$, we obtain
\[ \semimod_{\kk}(f): \semimod_{\kk}(X)\to \semimod_{\kk}(Y), \quad
 \sum_{x\in X}k_x[x] \mapsto  \sum_{x\in X}k_x[f(x)] \quad (k_x\in\kk). \]
Then, by Lemma \ref{lemm:semimod-monad}, we have a monad $\scrM_{\kk}=(\semimod_{\kk}, \eta^{\semimod_{\kk}}, \mu^{\semimod_{\kk}})$ on $\Set$ since each vector space is also a free $\kk$-semimodule.
Here:
\begin{itemize}
  \item $\eta^{\semimod_{\kk}}: \id_{\Set} \to \semimod_{\kk}$ is the natural transformation defined by $\eta_X^{\semimod_{\kk}}(x)=[x]$;
  \item and $\mu^{\semimod_{\kk}}: \semimod_{\kk}^2 \to \semimod_{\kk}$ is the natural transformation defined by
    \[\mu_X^{\semimod_{\kk}} \bigg(\sum\limits_{\ell=1}^{n}c_{\ell}
   \bigg[\sum\limits_{x\in X}a_{\ell,x}[x]\bigg]\bigg)
 = \sum\limits_{x\in X}\bigg(\sum\limits_{\ell=1}^{n}c_{\ell}a_{\ell,x}\bigg)[x].\]
\end{itemize}

\begin{definition}\rm
We call the monad $\scrM_{\kk}=(\semimod_{\kk}, \eta^{\semimod_{\kk}}, \mu^{\semimod_{\kk}})$ as above a \defines{free $\kk$-module monad} or a \defines{$\kk$-linear space monad}.
\end{definition}

By the above free $\kk$-module monad $\scrM_{\kk}=(\semimod_{\kk}, \eta^{\semimod_{\kk}}, \mu^{\semimod_{\kk}})$,
we obtain a Kleisli category $\Kl(\scrM_{\kk})$. Then, for finite sets $X$ and $Y$, a Kleisli arrow $f:X \rightsquigarrow Y$ in $\Kl(\scrM_{\kk})$ can be written as a function
\[ f(x)=\sum_{y\in Y}f_{y,x}[y], \quad x\in X. \]
Thus, $f$ is represented by the matrix $[f]=(f_{y,x})_{y\in Y,\ x\in X}$.
If $f:X\to\semimod_{\kk}(Y)$ and $h: Y\to\semimod_{\kk}(Z)$ are Kleisli arrows, then we have
\begin{align}\label{eq:linear-Kleisli-composition}
 [h\compos_{\Kl}f]=[h][f].
\end{align}
Therefore, Kleisli composition is ordinary matrix multiplication. See Lemma \ref{fact:Kl(scrM) is Mat}.

By Lemma \ref{lemm:monoid object of MatS}, the Kleisli category $\Kl(\scrM_{\kk})$ has a monoid object
\[(E_X:=X\times X, m_X:E_X\times E_X \rightsquigarrow E_X, u_X: \bfone \rightsquigarrow E_X)\]
where:
\begin{itemize}
  \item $m_X$, as a map $E_X\times E_X\longrightarrow\semimod_{\kk}(E_X)$, is defined by
\[ m_X((i,j),(p,q))=
 \begin{cases}
   [(i,q)], & j=p;\\
   0, & j\ne p;
 \end{cases}\]

  \item and $u_X$ is defined by $u_X(\bullet) = \sum\limits_{x\in X} [(x,x)]$.
\end{itemize}
Hence, the following result shows that $\KlHom_{\kk}(X) := \Hom_{\Kl(\scrM_{\kk})}(\bfone,E_X)$
is a $\kk$-algebra under Kleisli convolution.

\begin{lemma} \label{lemm:KlHom-alg}
Let $X$ be a finite set. The Kleisli Hom-space $\KlHom_{\kk}(X)$ is a $\kk$-algebra whose multiplication is given by the Kleisli convolution and whose unit is $u_X:\bfone\rightsquigarrow E_X$.
Furthermore, we have the following isomorphism of $\kk$-algebras
\[ \Phi_X:\KlHom_{\kk}(X)\longrightarrow\Mat_X(\kk), \quad \Phi_X(f)=(f_{i,j})_{i,j\in X},\]
where $f(\bullet)=\sum\limits_{(i,j)\in X\times X}f_{i,j}[(i,j)]$.
Under this isomorphism, Kleisli convolution is ordinary matrix multiplication.
\end{lemma}

\begin{proof}
Since $\bfone=\{\bullet\}$ is a singleton set, every $f\in\KlHom_{\kk}(X)$ ($=\Hom_{\Kl(\scrM_{\kk})}(\bfone,E_X)$) is uniquely determined by $f(\bullet)\in\semimod_{\kk}(E_X)$.
Therefore, $\KlHom_{\kk}(X)$ is a $\kk$-vector space since
\[ (f+g)(\bullet)=f(\bullet)+g(\bullet)
 \quad\text{and}\quad
 (\lambda f)(\bullet)=\lambda f(\bullet) \]
holds for all $f,g\in\KlHom_{\kk}(X)$ and $\lambda\in\kk$.
Here, its zero element is the Kleisli arrow sending $\bullet$ to the zero element of $\semimod_{\kk}(E_X)$.
Furthermore, applying Proposition \ref{prop:matrix-convolution} to $S=\kk$, we obtain an isomorphism of $\kk$-linear spaces
\[ \KlHom_{\kk}(X) = \Hom_{\Kl(\scrM_{\kk})}(\bfone,E_X) \cong \Mat_X(\kk).\]
Under this isomorphism, we have the correspondence $f \mapsto (f_{i,j})_{i,j\in X}$,
the Kleisli convolution corresponds to ordinary matrix multiplication, and $u_X$ corresponds to the identity matrix.
Therefore, $\KlHom_{\kk}(X)$ is a $\kk$-algebra whose multiplication is the Kleisli convolution and whose unit is $u_X$.
\end{proof}

We know that for each group $G$, $\kk G$ is a $\kk$-algebra.
In order to obtain a representation of $\kk G$, we need another $\kk$-algebra $\mathit{\Lambda}$
and establish a homomorphism from $\kk G$ to $\mathit{\Lambda}$.
By Lemma \ref{lemm:KlHom-alg}, $\KlHom_{\kk}(X)$ is a $\kk$-algebra whose multiplication is given by the convolution of Kleisli arrows, and then we can define a new representation of $G$.

\begin{definition}\rm \label{def:linear Kl convol repr gp}\
\begin{enumerate}[label=(\arabic*)]
  \item
Let $G$ be a group and let $\kk G$ be its group algebra.
A \defines{linear Kleisli convolution representation} of $G$ on a
finite set $X$ is a $\kk$-algebra homomorphism
\[ \rho_{\kk}: \kk G \to \KlHom_{\kk}(X) = \Hom_{\Kl(\scrM_{\kk})}(\bfone, E_X).\]
Equivalently, it can be seen as a pair $(X,\rho_{\kk})$ consisting of a finite set $X$ and a $\kk$-algebra homomorphism $\rho_{\kk}$ from $\kk G$ to $\KlHom_{\kk}(X)$.
\label{def:linear Kl convol repr gp 1}

  \item Let $(X,\rho_{\kk, X})$ and $(Y,\rho_{\kk, Y})$ be two linear Kleisli convolution representations of $G$.
A \defines{homomorphism} from $(X,\rho_{\kk,X})$ to $(Y,\rho_{\kk,Y})$ is a Kleisli arrow $f:X \to \semimod_{\kk}(Y)$ such that
\[  \Phi_Y(\rho_{\kk,Y}(a))[f] = [f]\Phi_X(\rho_{\kk,X}(a)) \quad\text{for all~}a\in\kk G. \]
\label{def:linear Kl convol repr gp 2}
\end{enumerate}
\end{definition}

The following result shows that we can construct a category whose objects are linear Kleisli convolution representations of $G$ and whose morphisms are the homomorphisms between two linear Kleisli convolution representations.

\begin{theorem}\label{thm:linear KLcrep gp}
Let $G$ be a group. All linear Kleisli convolution representations of $G$ and all homomorphisms between them form a category $\KCrep_{\kk}(G)$.
\end{theorem}

\begin{proof}
The objects of $\KCrep_{\kk}(G)$ are all linear Kleisli convolution representations $(X,\rho_{\kk,X})$ of $G$.
For two objects $(X,\rho_{\kk,X})$ and $(Y,\rho_{\kk,Y})$, the morphisms from $(X,\rho_{\kk,X})$ to $(Y,\rho_{\kk,Y})$ are the Kleisli arrows $f:X\rightsquigarrow Y$ satisfying
\[ \Phi_Y(\rho_{\kk,Y}(a))[f] = [f]\Phi_X(\rho_{\kk,X}(a)) \]
for every $a\in\kk G$.

We first define the composition of morphisms.
Let $f:(X,\rho_{\kk,X}) \to (Y,\rho_{\kk,Y})$ and $h:(Y,\rho_{\kk,Y}) \to (Z,\rho_{\kk,Z})$ be two homomorphisms.
We define their composition to be the Kleisli composition $h\compos_{\Kl}f:X\rightsquigarrow Z$.
By \eqref{eq:linear-Kleisli-composition}, its matrix is given by $[h\compos_{\Kl}f]=[h][f]$.
For every $a\in\kk G$, we have
\begin{align*}
 \Phi_Z(\rho_{\kk,Z}(a))[h\compos_{\Kl}f]
&=\Phi_Z(\rho_{\kk,Z}(a))[h][f]\\
&=[h]\Phi_Y(\rho_{\kk,Y}(a))[f]\\
&\mathop{=}\limits^{(\ast)}[h][f]\Phi_X(\rho_{\kk,X}(a))\\
&=[h\compos_{\Kl}f]\Phi_X(\rho_{\kk,X}(a)).
\end{align*}
Here, the second equality holds because $h$ is a homomorphism, and $(\ast)$ holds because $f$ is a homomorphism.
Therefore, $h\compos_{\Kl}f$ is a homomorphism from $(X,\rho_{\kk,X})$ to $(Z,\rho_{\kk,Z})$.
Thus, the composition of morphisms is well-defined.

For every object $(X,\rho_{\kk,X})$, consider the Kleisli arrow
$\eta_X^{\semimod_{\kk}}:X\rightsquigarrow X$ which sends each $x\in X$ to $[x]$.
It corresponds to the identity matrix $I_X$. It follows that, for every $a\in\kk G$,
\begin{align*}
 \Phi_X(\rho_{\kk,X}(a))[\eta_X^{\semimod_{\kk}}]
&= \Phi_X(\rho_{\kk,X}(a))I_X \\
&= I_X\Phi_X(\rho_{\kk,X}(a)) \\
&= [\eta_X^{\semimod_{\kk}}]\Phi_X(\rho_{\kk,X}(a)).
\end{align*}
Therefore, $\eta_X^{\semimod_{\kk}}$ is a homomorphism from
$(X,\rho_{\kk,X})$ to itself.

Finally, the composition of morphisms is associative because it is the composition in the Kleisli category $\Kl(\scrM_{\kk})$.
Indeed, for any three composable morphisms $f,h$ and $k$, we have $(k\compos_{\Kl}h)\compos_{\Kl}f= k\compos_{\Kl}(h\compos_{\Kl}f)$. Moreover, we obtain
\[ \eta_Y^{\semimod_{\kk}}\compos_{\Kl}f = f = f\compos_{\Kl}\eta_X^{\semimod_{\kk}}\]
for every homomorphism $f:(X,\rho_{\kk,X})\to(Y,\rho_{\kk,Y})$.

Thus, the objects, morphisms, compositions and identity morphisms defined above satisfy all the axioms of a category.
Therefore, they form the category $\KCrep_{\kk}(G)$.
\end{proof}

\begin{theorem}\label{theo:KC-linear}
There is an equivalence of categories $\KCrep_{\kk}(G)\simeq\rep_{\kk}(G)$.
\end{theorem}

\begin{proof}
By Theorem~\ref{thm:linear KLcrep gp}, $\KCrep_{\kk}(G)$ is a category.
Next, we construct a functor $\mathcal F_{\kk}:\KCrep_{\kk}(G) \to \rep_{\kk}(G)$.
Let $(X, \rho_{\kk,X})$ be an object of $\KCrep_{\kk}(G)$.
By Lemma \ref{lemm:KlHom-alg}, we have an isomorphism of $\kk$-algebras
$\Phi_X:\KlHom_{\kk}(X) \mathop{\to}\limits^{\cong} \Mat_X(\kk)$.
Therefore, the composition
${\Phi_X \compos \rho_{\kk,X}}:
 \kk G\To{\rho_{\kk,X}}\KlHom_{\kk}(X)
 \To{\Phi_X}\Mat_X(\kk)$
is a $\kk$-algebra homomorphism.
It defines a left $\kk G$-module $\kk  G \times \semimod_{\kk}(X) \to \semimod_{\kk}(X)$
given by $(a,v) \mapsto a\cdot v :=\Phi_X(\rho_{\kk,X}(a))v$.
Since $X$ is finite, $\semimod_{\kk}(X)$ is a finite-dimensional $\kk$-vector space.
Thus, we obtain a correspondence
\[\mathcal F_{\kk}(X,\rho_{\kk,X}) := \semimod_{\kk}(X)\]
sending each linear Kleisli convolution representation to a left $\kk G$-module $\semimod_{\kk}(X)$.
It is a \emph{correspondence of objects from $\KCrep_{\kk}(G)$ to $\rep_{\kk}(G)$.}

Let $f:(X,\rho_{\kk,X}) \to (Y,\rho_{\kk,Y})$ be a homomorphism in $\KCrep_{\kk}(G)$.
Thus, $f$ is a Kleisli arrow $f:X\rightsquigarrow Y$, which can be written as $f(x)=\sum\limits_{y\in Y}f_{y,x}[y]$.
Define
\[ \mathcal F_{\kk}(f):  \semimod_{\kk}(X) \to \semimod_{\kk}(Y), ~ [x] \mapsto f(x). \]
Its matrix with respect to the bases indexed by $X$ and $Y$ is $[f] = (f_{y,x})$.
By Definition \ref{def:linear Kl convol repr gp} \ref{def:linear Kl convol repr gp 2}, we have
$\Phi_Y(\rho_{\kk,Y}(a))[f] = [f]\Phi_X(\rho_{\kk,X}(a))$ ($\forall a\in\kk G$).
Then, for every $v\in\semimod_{\kk}(X)$, we have
$\mathcal F_{\kk}(f)(a\cdot v) = [f]\Phi_X(\rho_{\kk,X}(a))v = \Phi_Y(\rho_{\kk,Y}(a))[f]v = a\cdot\mathcal F_{\kk}(f)(v)$.
It follows that \emph{$\mathcal F_{\kk}(f)$ is a homomorphism of left $\kk G$-modules}.

By \eqref{eq:linear-Kleisli-composition}, for any two composable
morphisms $f$ and $h$, we have $[h\compos_{\Kl}f]=[h][f]$. Then we get
$\mathcal F_{\kk}(h\compos_{\Kl}f)= \mathcal F_{\kk}(h)\compos\mathcal F_{\kk}(f)$.
Moreover, the identity Kleisli arrow
$\eta_X^{\semimod_{\kk}}$ has the identity matrix $I_X$, and hence
$\mathcal F_{\kk}(\eta_X^{\semimod_{\kk}})
=\id_{\semimod_{\kk}(X)}$. Therefore, $\mathcal F_{\kk}$ is a
functor.

Let $f$ and $h$ be two morphisms from $(X,\rho_{\kk,X})$ to $(Y,\rho_{\kk,Y})$ such that $\mathcal F_{\kk}(f)=\mathcal F_{\kk}(h)$.
Then these two linear maps have the same matrix, so $[f]=[h]$.
A Kleisli arrow $X\rightsquigarrow Y$ is uniquely determined by its matrix.
Therefore, $f=h$, and so \emph{$\mathcal F_{\kk}$ is faithful}.

Let $\varphi: \mathcal F_{\kk}(X,\rho_{\kk,X}) \to \mathcal F_{\kk}(Y,\rho_{\kk,Y})$ be a homomorphism of left $\kk G$-modules.
Let $[\varphi]$ be its matrix with respect to the bases indexed by $X$ and $Y$.
Then there is a unique Kleisli arrow $f:X\rightsquigarrow Y$ whose matrix is $[f]=[\varphi]$.
Since $\varphi$ is $\kk G$-linear, for every $a\in\kk G$ we have
$ [f]\Phi_X(\rho_{\kk,X}(a)) = \Phi_Y(\rho_{\kk,Y}(a))[f]$.
Thus, $f$ is a morphism from $(X,\rho_{\kk,X})$ to $(Y,\rho_{\kk,Y})$ in $\KCrep_{\kk}(G)$, and $\mathcal F_{\kk}(f)=\varphi$.
Therefore, $\mathcal F_{\kk}$ is full.

Finally, we prove that $\mathcal F_{\kk}$ is essentially surjective.
Let $V$ be a finite-dimensional left $\kk G$-module.
Choose a basis $\{v_x:x\in X\}$ of $V$, indexed by a finite set $X$.
For every $a\in\kk G$, let $[a]_X\in\Mat_X(\kk)$ be the matrix of the linear map
$V\to V$, $v\mapsto a\cdot v$ with respect to this basis. The module axioms imply that the map
$[-]_X : a\mapsto[a]_X$ is a $\kk$-algebra homomorphism from $\kk G$ to $\Mat_X(\kk)$.
By Lemma \ref{lemm:KlHom-alg}, we can define
\[ \rho_{\kk,X} = \Phi_X^{-1}\compos [-]_X : \kk G \to \Mat_X(\kk) \to \KlHom_{\kk}(X), ~ a\mapsto \Phi_X^{-1}([a]_X). \]
Then it is a $\kk$-algebra homomorphism, and so $(X,\rho_{\kk,X})$ is an object of $\KCrep_{\kk}(G)$.
The linear map $\semimod_{\kk}(X)\to V$, $[x]\mapsto v_x$ is an isomorphism of $\kk G$-modules from $\mathcal F_{\kk}(X,\rho_{\kk,X})$ to $V$.
Therefore, every object of $\rep_{\kk}(G)$ is isomorphic to an object in the image of $\mathcal F_{\kk}$.
Then we obtain that $\mathcal F_{\kk}$ is essentially surjective.

Since $\mathcal F_{\kk}$ is full, faithful and essentially surjective, we have $\KCrep_{\kk}(G) \mathop{\simeq}\limits^{\mathcal F_{\kk}} \rep_{\kk}(G)$.
\end{proof}

\section{Rings and their Kleisli convolution representations}\label{sect:ring KCR}

In this section, we introduce Kleisli convolution representations of rings and compare their category with a full subcategory of the module category. We also discuss the corresponding Eilenberg--Moore category and some special representations.

\subsection{Rings and modules}

We first recall some basic notions of rings and modules, see \cite{Lam}.
Throughout this section, $R$ is an associative ring with identity $1_R$.
Recall that a \defines{left $R$-module} is an Abelian group $(M,+,0)$ together with a map
\[R\times M\to M,\quad (r,m)\mapsto rm\]
such that $(r+s)m=rm+sm$, $r(m+n)=rm+rn$, $(rs)m=r(sm)$ and $1_Rm=m$ hold for all $r,s\in R$ and $m,n\in M$.
In this paper, all modules are unital left modules.
For two $R$-modules $M$ and $N$, an \defines{$R$-homomorphism} (or equivalently, a \defines{homomorphism of $R$-modules}) is an Abelian group homomorphism $h: M \to N$ with
\[ h(r\cdot m)= r\cdot h(m), \quad \forall r\in R, m\in M. \]
We denote the category of left $R$-modules and $R$-module homomorphisms by $R\text{-}\Mod$,
and denote its full subcategory of finitely generated left $R$-modules by $R\text{-}\modcat$.

Note that a left $R$-module can also be defined by a ring homomorphism, i.e.,
it is a pair $(M, \varrho)$ of an Abelian group $(M,+,0)$ and a ring homomorphism
\[\varrho: R\to \End_{\ZZ}(M), \quad r\mapsto \varrho(r) := \widetilde{r}. \]
Under this definition, the $R$-homomorphism can be defined as an Abelian group homomorphism $h: M \to N$ with
\[ h(\widetilde{r}(m)) = \widetilde{r}(h(m)), \quad \forall r\in R, m\in M. \]

\subsection{Kleisli convolution representation}

Let $R\text{-}\mathsf{free}_{\ZZ}$ be the full subcategory of $R\text{-}\modcat$ consisting of the modules whose underlying Abelian groups are free of finite rank.
Notice that the notation $R\text{-}\mathsf{free}_{\ZZ}$ refers to freeness over $\ZZ$,
thus it does not mean the category of finitely generated free $R$-modules.

Applying Lemma \ref{lemm:semimod-monad} to $S=\ZZ$, we obtain a monad
$\scrM_{\ZZ}=(\semimod_{\ZZ}, \eta^{\semimod_{\ZZ}}, \mu^{\semimod_{\ZZ}})$,
which is called a \defines{free Abelian group monad}.
It implies that we can establish a Kleisli category $\Kl(\scrM_{\ZZ})$. By Proposition
\ref{prop:matrix-convolution}, for every finite set $X$ we have an isomorphism of rings
\[ \Phi_X:\KlHom_{\ZZ}(X) =\Hom_{\Kl(\scrM_{\ZZ})}(\bfone,E_X) \to \Mat_X(\ZZ). \]
Under this isomorphism, the Kleisli convolution is ordinary matrix multiplication and $u_X$ corresponds to the identity matrix $I_X$.

\begin{definition}\rm \label{def:Kl convol rep ring} \
\begin{enumerate}[label=(\arabic*)]
  \item A \defines{Kleisli convolution representation} of a ring $R$ on a finite set $X$ is a ring homomorphism
  \[ \varrho: R \to \KlHom_{\ZZ}(X).\]
  Equivalently, it is a pair $(X,\varrho)$ consisting of a finite set $X$
  and a unital ring homomorphism $\varrho$ as above.
  \label{def:Kl convol rep ring 1}

  \item Let $(X,\varrho_X)$ and $(Y,\varrho_Y)$ be two Kleisli convolution representations of $R$.
  A \defines{homomorphism} from $(X,\varrho_X)$ to $(Y,\varrho_Y)$ is a Kleisli arrow
  $f:X\rightsquigarrow Y$ in $\Kl(\scrM_{\ZZ})$ such that
  \[ \Phi_Y(\varrho_Y(r))[f]=[f]\Phi_X(\varrho_X(r)) \text{ for all } r\in R, \]
  where $[f]$ is the integer matrix associated with $f$.
  \label{def:Kl convol rep ring 2}
\end{enumerate}
\end{definition}

The following result shows that the above representations and their
homomorphisms form a category.

\begin{theorem}\label{thm:ring-KCrep-category}
All  Kleisli convolution representations of $R$ and all homomorphisms between them form an additive category $\KCrep_{\ZZ}(R)$.
\end{theorem}

\begin{proof}
Let $f:(X,\varrho_X)\to(Y,\varrho_Y)$ and $h:(Y,\varrho_Y)\to(Z,\varrho_Z)$ be two homomorphisms.
We define their composition to be the Kleisli composition $h\compos_{\Kl}f:X\rightsquigarrow Z$.
By Lemma \ref{fact:Kl(scrM) is Mat}, its matrix is $[h\compos_{\Kl}f]=[h][f]$.
For every $r\in R$, we have
\begin{align*}
 \Phi_Z(\varrho_Z(r))[h\compos_{\Kl}f]
 &=\Phi_Z(\varrho_Z(r))[h][f]\\
 &=[h]\Phi_Y(\varrho_Y(r))[f]\\
 &=[h][f]\Phi_X(\varrho_X(r))\\
 &=[h\compos_{\Kl}f]\Phi_X(\varrho_X(r)).
\end{align*}
Here, the equalities follow from Definition \ref{def:Kl convol rep ring} \ref{def:Kl convol rep ring 2}.
Thus, $h\compos_{\Kl}f$ is a homomorphism.

For every object $(X,\varrho_X)$, the Kleisli identity $\eta_X^{\semimod_{\ZZ}}:X\rightsquigarrow X$ corresponds to $I_X$.
Hence, for every $r\in R$, we have $\Phi_X(\varrho_X(r))I_X=I_X\Phi_X(\varrho_X(r))$,
and therefore $\eta_X^{\semimod_{\ZZ}}$ is a homomorphism from $(X,\varrho_X)$ to itself.
The associativity and the identity laws follow from those in the Kleisli category $\Kl(\scrM_{\ZZ})$.
Consequently, the objects and homomorphisms in Definition \ref{def:Kl convol rep ring} form a category $\KCrep_{\ZZ}(R)$.

Next, we show that $\KCrep_{\ZZ}(R)$ is an additive category.
If $f$ and $g$ are two homomorphisms from $(X,\varrho_X)$ to $(Y,\varrho_Y)$, then we have
\begin{align*}
 \Phi_Y(\varrho_Y(r))([f]+[g]) = \Phi_Y(\varrho_Y(r))[f]+\Phi_Y(\varrho_Y(r))[g] = ([f]+[g])\Phi_X(\varrho_X(r)).
\end{align*}
since $\Phi_Y:\KlHom_{\ZZ}(Y) \to \Mat_Y(\ZZ)$ is a ring homomorphism.
Thus, $f+g$ is a homomorphism of Kleisli convolution representations,
and every Kleisli Hom-space $\Hom_{\KCrep_{\ZZ}(R)}((X,\varrho_X), (Y,\varrho_Y))$ is an Abelian group.
Kleisli composition is bilinear because matrix multiplication is bilinear.
Finally, if $X\sqcup Y$ is the disjoint union, then the homomorphism
\[
 R \to \Mat_{X\sqcup Y}(\ZZ),\quad
 r \mapsto
 \begin{pmatrix}
  \Phi_X(\varrho_X(r))&0\\
  0&\Phi_Y(\varrho_Y(r))
 \end{pmatrix}
\]
defines a biproduct of $(X,\varrho_X)$ and $(Y,\varrho_Y)$.
The representation on the empty set is a zero object.
Then $\KCrep_{\ZZ}(R)$ is additive.
\end{proof}

\begin{theorem}\label{thm:ring-free}
There is an additive equivalence $\KCrep_{\ZZ}(R) \simeq R\text{-}\mathsf{free}_{\ZZ}$.
\end{theorem}

\begin{proof}
By Theorem \ref{thm:ring-KCrep-category}, $\KCrep_{\ZZ}(R)$ is an additive category.
Now we establish an additive functor $F_R$.
Let $(X,\varrho_X)$ be an object of $\KCrep_{\ZZ}(R)$, and define
\[F_R(X,\varrho_X)=\semimod_{\ZZ}(X). \]
By \eqref{defeq:semimod}, $\semimod_{\ZZ}(X)$ is the free Abelian group with basis $\{[x]:x\in X\}$.
Using the isomorphism $\Phi_X$ in Proposition \ref{prop:matrix-convolution},
define an action of $R$ by $r\cdot v=\Phi_X(\varrho_X(r))v$ $(r\in R,\ v\in\semimod_{\ZZ}(X))$.
Since $\varrho_X$ is a unital ring homomorphism and $\Phi_X$ is an isomorphism of rings,
for $r,s\in R$ and $v,w\in\semimod_{\ZZ}(X)$ we have
$(r+s)\cdot v=r\cdot v+s\cdot v,~
   r\cdot(v+w)=r\cdot v+r\cdot w,~
   (rs)\cdot v=r\cdot(s\cdot v),\text{ and }
   1_R\cdot v=v$.
Therefore, $\semimod_{\ZZ}(X)$ is a left $R$-module.
Its underlying Abelian group is free of finite rank.
Moreover, it is finitely generated as an $R$-module.
Indeed, every element has the form $\sum\limits_{x\in X}n_x[x]$,
and $\sum\limits_{x\in X}n_x[x] = \sum\limits_{x\in X}(n_x1_R)\cdot[x]$.
Thus, the elements $[x]$, $x\in X$, generate $\semimod_{\ZZ}(X)$ as an $R$-module.

Let $f\colon(X,\varrho_X)\to(Y,\varrho_Y)$ be a homomorphism, and define
\[ F_R(f):\semimod_{\ZZ}(X)\to \semimod_{\ZZ}(Y), \quad [x] \mapsto f(x).\]
Its matrix with respect to the bases indexed by $X$ and $Y$ is $[f]$.
For every $r\in R$ and $v\in\semimod_{\ZZ}(X)$, we have
$F_R(f)(r\cdot v) = [f]\Phi_X(\varrho_X(r))v = \Phi_Y(\varrho_Y(r))[f]v = r\cdot F_R(f)(v)$
by Definition \ref{def:Kl convol rep ring} \ref{def:Kl convol rep ring 2}.
Thus, $F_R(f)$ is $R$-linear. By Lemma \ref{fact:Kl(scrM) is Mat},
Kleisli composition corresponds to matrix multiplication and the Kleisli identity corresponds to the identity matrix.
Thus, $F_R$ preserves compositions and identities. It also preserves sums of homomorphisms.
Therefore, $F_R$ is an additive functor.

Suppose that two homomorphisms $f$ and $g$ have the same image under $F_R$.
Then they have the same values on every basis element $[x]$, i.e., $f(x)=g(x)$ for every $x\in X$, and so $f=g$.
Therefore, $F_R$ is faithful.

Let $h: F_R(X,\varrho_X)\to F_R(Y,\varrho_Y)$ be an $R$-module homomorphism.
Since the source and target are free Abelian groups with the specified bases,
$h$ is represented by a unique integer matrix $[f]$, i.e., there is a unique Kleisli arrow $f:X\rightsquigarrow Y$ corresponding to the matrix $[f]$.
The equality $h(r\cdot v)=r\cdot h(v)$ is equivalent to $[f]\Phi_X(\varrho_X(r))=\Phi_Y(\varrho_Y(r))[f]$
for all $r\in R$. Therefore, $f$ is a homomorphism in $\KCrep_{\ZZ}(R)$, and $F_R(f)=h$.
Thus, $F_R$ is full.

Finally, let $M$ be an object of $R\text{-}\mathsf{free}_{\ZZ}$.
Choose a $\ZZ$-basis $\{m_x\mid x\in X\}$ of its underlying Abelian group.
For every $r\in R$, the endomorphism $h_r: M\to M,~ m\mapsto r m$ has an integer matrix $\sigma_M(r)$,
and for any $r_1, r_2\in R$, we have
\[ \sigma_M(r_1+r_2)=\sigma_M(r_1)+\sigma_M(r_2),~
   \sigma_M(r_1r_2)=\sigma_M(r_1)\sigma_M(r_2),~\text{ and }
   \sigma_M(1_R)=I_X. \]
Then we obtain a ring homomorphism
$\sigma_M: R \to \Mat_{X}(\ZZ)$ ($\mathop{\longrightarrow}\limits_{\cong}^{\varsigma_M} \End_{\ZZ}(M)$),
$r\mapsto \sigma_M(r) := \varsigma^{-1}_M(h_r)$.
By Proposition \ref{prop:matrix-convolution}, define $\varrho_M=\Phi_X^{-1}\compos\sigma_M: R\to\KlHom_{\ZZ}(X)$.
Then $\varrho_M$ is a unital ring homomorphism, and so $(X,\varrho_M)$ is an object of $\KCrep_{\ZZ}(R)$.
The map $\semimod_{\ZZ}(X)\to M$, $[x]\mapsto m_x$ is an $R$-module isomorphism $F_R(X,\varrho_M)\cong M$.
Thus, $F_R$ is essentially surjective. Since it is full, faithful and essentially surjective,
we get that it is an additive equivalence.
\end{proof}

In general, Theorem \ref{thm:ring-free} cannot be extended to all finitely generated modules without changing the construction,
see the following example.

\begin{example}\label{ex:ring-obstruction} \rm
Take $R=\ZZ$. The finitely generated $\ZZ$-module $\ZZ/2\ZZ$ cannot be isomorphic to $\ZZ^{\oplus X}$ for any finite set $X$.
Therefore, it is not represented by an object of $\KCrep_{\ZZ}(\ZZ)$.
In this case, we have $\KCrep_{\ZZ}(\ZZ) \not\simeq\ZZ\text{-}\modcat$.
\end{example}

\begin{remark} \rm
If one restricts attention to rings and modules for which the relevant underlying Abelian groups are free, then the matrix Kleisli description is enough. Over a field, every vector space is free, and this is why the algebra case in Section \ref{sect:algebra KCR} gives the whole finite-dimensional module category.
\end{remark}

\subsection{Vs. Eilenberg--Moore categories}

There is nevertheless a monadic description of the whole module category.
This is an elementary instance of the \defines{Eilenberg--Moore construction} given in \cite{EM1965}.
Consider the endofunctor
\[ T_R=R\otimes_{\ZZ}-\colon\Ab\to\Ab. \]
The unit is
\[ \eta_M^{T_R}:  M\to R\otimes_{\ZZ}M, \quad m\mapsto1_R\otimes m,\]
and the multiplication is induced by the multiplication of $R$ as follows
\[ \mu_M^{T_R}: R\otimes_{\ZZ}R\otimes_{\ZZ}M \to R\otimes_{\ZZ}M,\quad r\otimes s\otimes m\mapsto rs\otimes m. \]

\begin{lemma}\label{lemm:ring tensor monad}
The triple $\scrT_R=(T_R,\eta^{T_R},\mu^{T_R})$ is a monad on the category $\Ab$.
\end{lemma}

\begin{proof}
First, it is well-known that $T_R=R\otimes_{\ZZ}-: \Ab\to \Ab$ is an endofunctor on $\Ab$,
which:
\begin{itemize}
  \item sends each Abelian group $M$ to a tensor $R\otimes_{\ZZ} M$,
where the tensor product $R\otimes_{\ZZ} M$ is also an Abelian group;
  \item and sends each Abelian group homomorphism $f:M\to N$ to the map $R\otimes_{\ZZ}f:= \id_R\otimes f : R\otimes_{\ZZ}M \to R\otimes_{\ZZ}N$, $r\otimes m \mapsto r\otimes f(m)$,
where $R\otimes_{\ZZ}f$ is also an Abelian group homomorphism.
\end{itemize}

For an arbitrary homomorphism $f:M\to N$ and $m\in M$,
we have $T_R(f)(\eta^{T_R}_M(m))=1_R\otimes f(m)=\eta^{T_R}_N(f(m))$.
Thus, $\eta^{T_R}:\id_{\Ab}\to T_R$ is a natural transformation.
Moreover, for $r,s\in R$ and $m\in M$, both sides of $T_R(f)\compos\mu_M^{T_R} = \mu_N^{T_R} \compos T_R^2(f)$ send $r\otimes s\otimes m$ to $rs\otimes f(m)$.
Then $\mu^{T_R}:T_R^2\to T_R$ is also a natural transformation.

For $r,s,t\in R$ and $m\in M$, the two composites $\mu^{T_R}_M\compos T_R(\mu^{T_R}_M)$ and $\mu^{T_R}_M\compos\mu^{T_R}_{T_R(M)}$
send $r\otimes s\otimes t\otimes m$ to $r(st)\otimes m$ and $(rs)t\otimes m$, respectively.
They are equal since multiplication in $R$ is associative. Furthermore, we have
$\mu^{T_R}_M\compos\eta^{T_R}_{T_R(M)}(r\otimes m) = 1_Rr\otimes m=r\otimes m$ and
$\mu^{T_R}_M\compos T_R(\eta^{T_R}_M)(r\otimes m) = r1_R\otimes m=r\otimes m$.
Therefore, the associativity and two unit axioms hold, and so $\scrT_R$ is a monad on $\Ab$.
\end{proof}

We recall the definition of Eilenberg--Moore categories now.
This construction was introduced by Eilenberg and Moore in \cite[Section 2, pp.~383--385]{EM1965};
see also \cite[Chapter VI, Section 2, pp.~139--142]{Mac1998} and \cite[Section 5.2]{Riehl}.
Let $\scrT=(T,\eta,\mu)$ be a monad on a category $\calC$.
A \defines{$\scrT$-algebra} is a pair $(X,\alpha)$ consisting of an object $X$ of $\calC$ and a morphism $\alpha:T(X)\to X$ such that $\alpha\compos\eta_X=\id_X$ and $\alpha\compos T(\alpha)=\alpha\compos\mu_X$.
Let $(X,\alpha)$ and $(Y,\beta)$ be two $\scrT$-algebras.
A \defines{homomorphism of $\scrT$-algebras} from $(X,\alpha)$ to $(Y,\beta)$ is a morphism
$f:X\to Y$ in $\calC$ satisfying $f\compos\alpha=\beta\compos T(f)$.
The composition and identity morphisms are inherited from $\calC$.

\begin{definition}\rm
All $\scrT$-algebras and their homomorphisms form a category, denoted
by $\calC^{\scrT}$, which is called the \defines{Eilenberg--Moore category} of $\scrT$.
\end{definition}

\begin{remark}\rm
The notion of $\scrT$-algebras can be regarded as a categorical generalization of algebras over commutative rings.
To be precise, let $R$ be a commutative ring and define an endofunctor
$\calT_R:R\text{-}\Mod\to R\text{-}\Mod$ by
\[ \calT_R(M)=\bigoplus_{n\>= 0}M^{\otimes_R n}, \quad M^{\otimes_R 0}=R.\]
In this case, for an $R$-module homomorphism $f:M\to N$, $\calT_R(f)$ is $\bigoplus\limits_{n\>=0}f^{\otimes_R n}$.
The unit $\eta_M:M\to\calT_R(M)$ is the canonical inclusion into the component $M^{\otimes_R1}=M$.
The multiplication $\mu_M:\calT_R^2(M)\to \calT_R(M)$ is given by
\begin{align*}
  \mu: & (m_{11}\otimes\cdots\otimes m_{1n_1})\otimes\cdots\otimes
 (m_{\ell1}\otimes\cdots\otimes m_{\ell n_\ell})  \\
 & \mapsto m_{11}\otimes\cdots\otimes m_{1n_1}\otimes\cdots\otimes
 m_{\ell1}\otimes\cdots\otimes m_{\ell n_\ell}.
\end{align*}
Then $\scrT = (\calT_R,\eta,\mu)$ is a monad on $R\text{-}\Mod$,
cf. \cite[Chapter VI, Section 2]{Mac1998} and \cite[Section 5.2]{Riehl},
and the $\scrT$-algebras induced by $\scrT$, also called \defines{Eilenberg--Moore algebras},
are precisely the associative unital $R$-algebras.
Indeed, the structure map $\alpha:\calT_R(A)\to A$ associated with an $R$-algebra $A$ is given by
$\alpha(a_1\otimes\cdots\otimes a_n)=a_1\cdots a_n$.
Its restriction to $A^{\otimes_R2}$ gives the multiplication of $A$,
while its restriction to $A^{\otimes_R0}=R$ gives the unit map $R\to A$.
These two maps determine $\alpha$.
Therefore, ordinary associative unital $R$-algebras are special examples of $\scrT$-algebras.
\end{remark}

The following standard result identifies the Eilenberg--Moore category
of $\scrT_R$ with the category of left $R$-modules, cf. \cite[Chap 5, Example 5.2.6 {\rm(ii)}]{Riehl}.

\begin{proposition}\label{prop:em-ring}
The Eilenberg--Moore category $\Ab^{\scrT_R}$ is equivalent to $R\text{-}\Mod$.
\end{proposition}

The proof of Proposition \ref{prop:em-ring} is short, but we include it for the reader's convenience.

\begin{proof}
Notice that a $\scrT_R$-algebra is an Abelian group $M$ together with a homomorphism $\alpha\colon R\otimes_{\ZZ}M\to M$ satisfying $\alpha(1\otimes m)=m$ and $\alpha(r\otimes\alpha(s\otimes m))=\alpha(rs\otimes m)$.
Define $rm:=\alpha(r\otimes m)$. The above two equations imply $1_Rm=\alpha(1_R\otimes m)=m$
and $r(sm) = \alpha(r\otimes\alpha(s\otimes m)) = \alpha(rs\otimes m) = (rs)m$.
Since $\alpha$ is a homomorphism of Abelian groups and is defined on the tensor product,
this implies that $(r+s)m=rm+sm$ and $r(m+n)=rm+rn$ hold for all $r,s\in R$ and $m$, $n\in M$.
Thus, $M$ is a left $R$-module.

Next, let $M$ be a left $R$-module. The map $R\times M\to M$, $(r,m)\mapsto rm$ is biadditive.
By the universal property of the tensor product, it induces a unique homomorphism
$\alpha_M\colon R\otimes_{\ZZ}M\to M$, $r\otimes m \mapsto rm$,
which implies $1_Rm=m$ and $r(sm)=(rs)m$.

Finally, let $(M,\alpha_M)$ and $(N,\alpha_N)$ be two $\scrT_R$-algebras.
A homomorphism of Abelian groups $f: M\to N$ is a $\scrT_R$-algebra homomorphism precisely when $f\compos\alpha_M = \alpha_N\compos(1_R\otimes f)$.
Then we have $f\compos\alpha_M (r\otimes m) = \alpha_N\compos(1_R\otimes f) (r\otimes m)$,
and then we obtain $f(rm)=rf(m)$.
Therefore, the $\scrT_R$-algebra morphisms are exactly the $R$-module homomorphisms.
The two constructions above are mutually inverse on both objects and morphisms, proving $\Ab^{\scrT_R} \simeq R\text{-}\Mod$.
\end{proof}

\begin{remark}\label{rem:kleisli-em} \rm
The Kleisli category $\Kl(\scrT_R)$ and the Eilenberg--Moore category $\Ab^{\scrT_R}$ have different roles.
The former is generated by free $\scrT_R$-algebras, while the latter contains all $R$-modules.
In fact, $\Hom_{\Kl(\scrT_R)}(\ZZ,\ZZ) = \Hom_{\Ab}(\ZZ,R\otimes_{\ZZ}\ZZ) \cong R$.
Under evaluation at $1$, Kleisli composition identifies this endomorphism ring with $R^{\mathrm{op}}$. Indeed, if $f_r(1)=r\otimes1$ and $f_s(1)=s\otimes1$, then $(f_s\compos_{\Kl}f_r)(1)=rs\otimes1$.
However, this endomorphism calculation does not imply that $\Kl(\scrT_R)$ is equivalent to $R\text{-}\Mod$.
This category is different from the construction $\KCrep_{\ZZ}(R)$ considered in Example \ref{ex:ring-obstruction}. In particular, when $R=\ZZ$, the monad $\scrT_R$ is isomorphic to the identity monad, and $\Kl(\scrT_R)\simeq\Ab\simeq R\text{-}\Mod$.
\end{remark}


\begin{corollary}\label{coro:em-ring}
There is a fully faithful functor $\KCrep_{\ZZ}(R)\to \Ab^{\scrT_R}$.
Its essential image consists precisely of those $\scrT_R$-algebras whose underlying abelian groups are free of finite rank.
\end{corollary}


\begin{proof}
By Theorem \ref{thm:ring-free}, we have an equivalence $\KCrep_{\ZZ}(R)\simeq R\text{-}\mathsf{free}_{\ZZ}$.
Moreover, $R\text{-}\mathsf{free}_{\ZZ}$ is a full subcategory of $R\text{-}\Mod$.
Combining this inclusion with the equivalence $\Ab^{\scrT_R}\simeq R\text{-}\Mod$ of categories,
given by Proposition \ref{prop:em-ring}, we obtain the required fully faithful functor.
Its essential image is precisely the stated subcategory.
\end{proof}

\subsection{Special Kleisli convolution representations of rings}

We now consider some special objects in $\KCrep_{\ZZ}(R)$. By Theorem~\ref{thm:ring-free},
every object $(X,\varrho)$ determines the left $R$-module $M_{\varrho}=\ZZ^{(X)}$ with action $r\cdot v=\Phi_X(\varrho(r))v$.
Hence the usual module-theoretic properties can be used to define special Kleisli convolution representations.
However, the underlying Abelian group of $M_{\varrho}$ is always free of finite rank.

\subsubsection{Flat Kleisli convolution representations}

Recall that a left $R$-module $M$ is \defines{flat} if the functor $-\otimes_R M$ from right $R$-modules to Abelian groups is exact.
Equivalently, for every monomorphism $u: N'\to N$ of right $R$-modules,
the induced homomorphism $u\otimes_R1_M\colon N'\otimes_RM\to N\otimes_RM$ is a monomorphism.

\begin{definition}\rm
An object $(X,\varrho)$ of $\KCrep_{\ZZ}(R)$ is called a \defines{flat Kleisli convolution representation} if the associated left $R$-module $M_{\varrho}$ is flat.
We denote the full subcategory consisting of such objects by $\KCrep_{\ZZ}^{\mathrm{flat}}(R)$.
\end{definition}

\begin{proposition}\label{prop:ring-flat}
The equivalence in Theorem~\ref{thm:ring-free} restricts to an additive equivalence
\[ \KCrep_{\ZZ}^{\mathrm{flat}}(R) \simeq (R\text{-}\mathsf{free}_{\ZZ})_{\mathrm{flat}},\]
where the category on the right is the full subcategory of flat modules.
\end{proposition}

\begin{proof}
Flatness is invariant under isomorphisms of $R$-modules.
The functor $F_R$ constructed in Theorem \ref{thm:ring-free} sends $(X,\varrho)$ to $M_{\varrho}$ and is
full, faithful and essentially surjective.
Therefore, its restriction has precisely the flat modules in $R\text{-}\mathsf{free}_{\ZZ}$ as its essential image.
The restricted functor is still full and faithful, and hence it is an equivalence.
\end{proof}

Every projective module is flat.
Hence, if the module $M_{\varrho}$ associated with $(X,\varrho)$ is projective,
then $(X,\varrho)$ is a flat Kleisli convolution representation.
The preceding proposition does not say that every flat $R$-module is represented by a finite Kleisli matrix,
because its underlying Abelian group must also be free of finite rank.

\subsubsection{Injective Kleisli convolution representations}

\begin{definition} \rm
An object $(X,\varrho)$ of $\KCrep_{\ZZ}(R)$ is called an
\defines{injective Kleisli convolution representation} if $M_{\varrho}$ is
injective in $R\text{-}\Mod$.
\end{definition}

\begin{proposition}\label{prop:ring-injective-zero}
There is no nonzero injective Kleisli convolution representation of $R$.
More precisely, if $M$ is an injective left $R$-module and its underlying Abelian group is free of finite rank, then $M=0$.
\end{proposition}

\begin{proof}
Suppose that $M \ne 0$ and that its underlying Abelian group is free of finite rank.
Fix an integer $n\>= 2$. Since $M$ is torsion-free, every element of $nM$ has a unique expression $nm$ with $m\in M$.
Define $f: nM\to M$, $nm\mapsto m$. The subgroup $nM$ is an $R$-submodule,
because $r(nm)=n(rm)$ for all $r\in R$. Moreover, $f(r(nm))=f(n(rm))=rm=rf(nm)$, so $f$ is $R$-linear.
Let $\iota: nM\hookrightarrow M$ be the inclusion.
If $M$ is injective, then the homomorphism $f$ would extend along $\iota$ to an $R$-module homomorphism $g: M\to M$.
For every $m\in M$ we would then have $m=f(nm)=g(nm)=ng(m)$.
Thus, every element of $M$ would be divisible by $n$.
A nonzero free Abelian group of finite rank is not divisible by $n$,
because a basis element does not belong to $nM$. This is a contradiction.
\end{proof}

\subsubsection{Simple Kleisli convolution representations}

\begin{definition} \rm
A nonzero object $(X,\varrho)$ of $\KCrep_{\ZZ}(R)$ is called a \defines{simple Kleisli convolution representation} if $M_{\varrho}$ is a simple left $R$-module.
\end{definition}

\begin{proposition}\label{prop:ring-simple-zero}
There is no simple Kleisli convolution representation of $R$.
\end{proposition}

\begin{proof}
Let $(X,\varrho)$ be nonzero and put $M=M_{\varrho}$.
For an integer $n \>= 2$, the subgroup $nM$ is an $R$-submodule of $M$.
Since the underlying Abelian group of $M$ is free of positive finite rank, it is torsion-free,
and hence $nM \ne 0$. On the other hand, $nM\ne M$, because a basis element of $M$ does not belong to $nM$.
Thus, $0\subsetneq nM\subsetneq M$ is a proper nonzero $R$-submodule. Therefore, $M$ is not simple.
\end{proof}

Propositions~\ref{prop:ring-injective-zero} and \ref{prop:ring-simple-zero} describe a genuine limitation of the matrix Kleisli construction. To include arbitrary injective and simple $R$-modules, one has to use the Eilenberg--Moore category $\Ab^{\scrT_R} \simeq R\text{-}\Mod$ from Proposition \ref{prop:em-ring}.

\section{Algebras and their Kleisli convolution representations}\label{sect:algebra KCR}

\subsection{Algebras and their representations}

We first recall some basic notions of algebras and their representations, see \cite{ARS,Lam}.
Let $\kk$ be a field. A \defines{$\kk$-algebra} $A$ is a ring together with a $\kk$-vector space structure such that
$\lambda(ab)=(\lambda a)b=a(\lambda b)$ holds for all $\lambda\in\kk$ and $a,b\in A$.
In this section, $A$ is a finite-dimensional associative $\kk$-algebra with identity $1_A$.

Let $V$ be a finite-dimensional $\kk$-vector space. A \defines{representation} of $A$ on $V$ is a unital $\kk$-algebra
homomorphism
\[ \omega : A \to \End_{\kk}(V). \]
We also write this representation as a pair $(V,\omega)$.
Equivalently, it is a finite-dimensional left $A$-module whose action is given by {$a\cdot v = \omega(a)(v)$}.
Let $(V,\omega_V)$ and $(W,\omega_W)$ be two representations of $A$.
A homomorphism from $(V,\omega_V)$ to $(W,\omega_W)$ is a $\kk$-linear map $h:V\to W$ satisfying
\[ h\compos\omega_V(a) = \omega_W(a)\compos h \quad\text{for every }a\in A. \]
Thus, these homomorphisms are precisely the homomorphisms of left $A$-modules.
The usual composition of linear maps and the identity maps make all finite-dimensional representations of $A$ into a category,
which is denoted by $\rep(A)$. Notice that the assumption that $\kk$ is algebraically closed is not needed for the constructions below.

\subsection{Kleisli convolution representations of algebras}

For a finite set $X$, Proposition~\ref{prop:matrix-convolution} gives
\[ \KlHom_{\kk}(X)=\Hom_{\Mat_{\kk}}(\bfone,E_X) \cong \Mat_X(\kk). \]

\begin{definition}\rm \label{def:Kl convol rep algebra} \
\begin{enumerate}[label=(\arabic*)]
  \item A \defines{Kleisli convolution representation} of $A$ on a finite set $X$ is a unital $\kk$-algebra homomorphism
  \[ \theta: A\longrightarrow\KlHom_{\kk}(X). \]
  Equivalently, it is a pair $(X,\theta)$ consisting of a finite set $X$ and a unital $\kk$-algebra homomorphism $\theta$ as above.
  \label{def:Kl convol rep algebra 1}

  \item Let $(X,\theta_X)$ and $(Y,\theta_Y)$ be two Kleisli convolution
  representations of $A$. A \defines{homomorphism} from $(X,\theta_X)$ to
  $(Y,\theta_Y)$ is a Kleisli arrow $f:X\rightsquigarrow Y$ in
  $\Kl(\scrM_{\kk})$ such that
  \[ \Phi_Y(\theta_Y(a))[f]=[f]\Phi_X(\theta_X(a)) \quad\text{for every }a\in A.  \]
  \label{def:Kl convol rep algebra 2}
\end{enumerate}
\end{definition}

\begin{theorem}\label{thm:algebra-KCrep-category}
All Kleisli convolution representations of $A$ and all homomorphisms between them form a $\kk$-linear category $\KCrep(A)$.
\end{theorem}

\begin{proof}
Let $f:(X,\theta_X)\to(Y,\theta_Y)$ and $h:(Y,\theta_Y)\to(Z,\theta_Z)$ be two homomorphisms.
Define their composition to be the Kleisli composition $h\compos_{\Kl}f:X\rightsquigarrow Z$.
By \eqref{eq:linear-Kleisli-composition}, we have $[h\compos_{\Kl}f]=[h][f]$.
Then, for every $a\in A$, we have
\begin{align*}
 \Phi_Z(\theta_Z(a))[h\compos_{\Kl}f]
 &=\Phi_Z(\theta_Z(a))[h][f]\\
 &=[h]\Phi_Y(\theta_Y(a))[f]\\
 &=[h][f]\Phi_X(\theta_X(a))\\
 &=[h\compos_{\Kl}f]\Phi_X(\theta_X(a)).
\end{align*}
Thus, $h\compos_{\Kl}f$ is a homomorphism from $(X,\theta_X)$ to $(Z,\theta_Z)$.

For every object $(X,\theta_X)$, the Kleisli identity $\eta_X^{\semimod_{\kk}}:X\rightsquigarrow X$ has the identity matrix
$I_X$. Therefore, we have $\Phi_X(\theta_X(a))I_X=I_X\Phi_X(\theta_X(a))$ for every $a\in A$,
and so $\eta_X^{\semimod_{\kk}}$ is a homomorphism from $(X,\theta_X)$ to itself.
The associativity and identity laws follow from those in $\Kl(\scrM_{\kk})$. Thus, $\KCrep(A)$ is a category.

Let $f$ and $g$ be two homomorphisms from $(X,\theta_X)$ to $(Y,\theta_Y)$ and let $\lambda\in\kk$.
Direct calculation implies $\Phi_Y(\theta_Y(a))([f]+[g])=([f]+[g])\Phi_X(\theta_X(a))$ and
$\Phi_Y(\theta_Y(a))(\lambda[f]) =(\lambda[f])\Phi_X(\theta_X(a))$.
Then every Kleisli Hom-space is a $\kk$-linear space.
Since matrix multiplication is bilinear, the composition is $\kk$-bilinear.
Therefore, $\KCrep(A)$ is a $\kk$-linear category.
\end{proof}

\begin{theorem}\label{thm:algebra-equivalence}
Let $A$ be a finite-dimensional $\kk$-algebra. Then there is a $\kk$-linear equivalence $\KCrep(A) \simeq \rep(A)$.
\end{theorem}

\begin{proof}
By Theorem \ref{thm:algebra-KCrep-category}, $\KCrep(A)$ is a $\kk$-linear category.
Now, we need to establish a functor $F_A: \KCrep(A)\to \rep(A)$ from $\KCrep(A)$ to $\rep(A)$.
Let $(X,\theta_X)$ be an object of $\KCrep(A)$, and define
\[ F_A(X,\theta_X)=\semimod_{\kk}(X). \]
By \eqref{defeq:semimod}, this is a finite-dimensional $\kk$-vector space with basis $\{[x]:x\in X\}$.
Let $V=\semimod_{\kk}(X)$. By Lemma \ref{lemm:KlHom-alg}, define an $A$-action $A\times V \to V$ by $a\cdot v=\Phi_X(\theta_X(a))v$.
Since $\theta_X$ and $\Phi_X$ are unital $\kk$-algebra homomorphisms,
we have $(a+b)\cdot v=a\cdot v+b\cdot v$, $a\cdot(v+w)=a\cdot v+a\cdot w$,
$(\lambda a)\cdot v=\lambda(a\cdot v)$, $(ab)\cdot v=a\cdot(b\cdot v)$, and $1_A\cdot v =v$
for all $\lambda\in\kk$, $a,b\in A$, $v,w\in V$.
Thus, $F_A(X,\theta_X)$ is a finite-dimensional left $A$-module.

Let $f\colon(X,\theta_X)\to(Y,\theta_Y)$ be a homomorphism.
Define $F_A(f):\semimod_{\kk}(X) \to \semimod_{\kk}(Y)$, $[x]\mapsto f(x)$.
Its matrix with respect to the bases indexed by $X$ and $Y$ is $[f]$.
For $a\in A$ and $v\in\semimod_{\kk}(X)$, we have
\begin{align*}
   F_A(f)(a\cdot v)
&= [f]\Phi_X(\theta_X(a))v \\
&= \Phi_Y(\theta_Y(a))[f]v \\
&= a\cdot F_A(f)(v)
\end{align*}
by Definition \ref{def:Kl convol rep algebra} \ref{def:Kl convol rep algebra 2}.
Thus, $F_A(f)$ is $A$-linear. By \eqref{eq:linear-Kleisli-composition},
$F_A$ preserves compositions and identities.
It also preserves addition and scalar multiplication of homomorphisms.
Thus, $F_A$ is a $\kk$-linear functor.

Let $f$ and $g$ be two homomorphisms from $(X,\theta_X)$ to $(Y,\theta_Y)$ such that $F_A(f)=F_A(g)$.
Then $[f]=[g]$. A Kleisli arrow in $\Kl(\scrM_{\kk})$ is uniquely determined by its matrix, and then we have $f=g$.
It follows that $F_A$ is faithful.

Take an $A$-linear map $h\colon F_A(X,\theta_X)\to F_A(Y,\theta_Y)$.
Let $[h]$ be its matrix with respect to the bases indexed by $X$ and $Y$.
There is a unique Kleisli arrow $f:X\rightsquigarrow Y$ such that $[f]=[h]$.
Since $h$ is $A$-linear, for every $a\in A$ we have $[f]\Phi_X(\theta_X(a))=\Phi_Y(\theta_Y(a))[f]$.
Thus, by Definition \ref{def:Kl convol rep algebra} \ref{def:Kl convol rep algebra 2},
$f$ is a homomorphism in $\KCrep(A)$, and $F_A(f)=h$.
It follows that $F_A$ is full.

Finally, let $M$ be a finite-dimensional left $A$-module.
Choose a basis $\{m_x\mid x\in X\}$ of $M$.
The action of $A$ gives a unital algebra homomorphism $\sigma_M:A\to \Mat_X(\kk)$,
where $\sigma_M(a)$ is the matrix of the linear map $m\mapsto a\cdot m$.
Thus, we have $\sigma_M(a+b)=\sigma_M(a)+\sigma_M(b)$,
$\sigma_M(\lambda a)=\lambda\sigma_M(a)$, $\sigma_M(ab)=\sigma_M(a)\sigma_M(b)$, and $\sigma_M(1_A)=I_X$.
By Lemma \ref{lemm:KlHom-alg}, define $\theta_M=\Phi_X^{-1}\compos\sigma_M: A \to \KlHom_{\kk}(X)$.
Then $\theta_M$ is a unital $\kk$-algebra homomorphism,
and so $(X,\theta_M)$ is an object of $\KCrep(A)$.
The linear isomorphism $\theta: \semimod_{\kk}(X)\to M$, $[x] \mapsto m_x$
satisfies $\theta(\Phi_X(\theta_M(a))v)=a\cdot\theta(v)$ for all $a\in A$ and $v\in\semimod_{\kk}(X)$.
Hence it is a left $A$-module isomorphism $F_A(X,\theta_M)\cong M$.
This proves that $F_A$ is essentially surjective.
Since it is also full and faithful, $F_A$ is a $\kk$-linear equivalence.
\end{proof}

We immediately have the following two corollaries.

\begin{corollary}\label{cor:algebra-KCrep-Abelian}
The category $\KCrep(A)$ is a Hom-finite Abelian $\kk$-category. It is also a Krull--Schmidt category.
\end{corollary}

\begin{proof}
By Theorem \ref{thm:algebra-equivalence}, $\KCrep(A)$ is equivalent to $\rep(A)$.
It is well-known that $\rep(A)$ is both a Hom-finite Abelian $\kk$-category and a Krull--Schmidt category (since $A$ is a finite-dimensional algebra).
\end{proof}

\begin{corollary}\label{cor:group-from-algebra}
For every group $G$, we have $\KCrep_{\kk}(G) \simeq \rep(\kk G)$, where $\rep(\kk G)$ denotes the category of finite-dimensional left $\kk G$-modules.
\end{corollary}

\begin{proof}
This follows immediately from Theorem \ref{theo:KC-linear}. Notice that
the group algebra $\kk G$ need not be finite-dimensional; only the represented modules are required to be finite-dimensional.
\end{proof}

\subsection{Special Kleisli convolution representations}

We next describe projective, injective and simple objects in $\KCrep(A)$.
Since Theorem~\ref{thm:algebra-equivalence} is an equivalence with the
Abelian category $\rep(A)$, the three usual classes of finite-dimensional
$A$-modules are completely represented in the Kleisli convolution category.

\subsubsection{Projective Kleisli convolution representations}

\begin{definition} \rm
An object $(X,\theta)$ of $\KCrep(A)$ is called a \defines{projective Kleisli convolution representation} if
$F_A(X,\theta)$ is a projective left $A$-module.
We denote the full subcategory consisting of all such objects by $\KCrep^{\mathrm{proj}}(A)$.
\end{definition}

\begin{proposition}\label{prop:algebra-projective}
The equivalence given in Theorem \ref{thm:algebra-equivalence} restricts to an equivalence
\[ \KCrep^{\mathrm{proj}}(A)\simeq\mathsf{proj}(A), \]
where $\mathsf{proj}(A)$ is the category of finite-dimensional projective left $A$-modules.
\end{proposition}

\begin{proof}
An equivalence preserves and reflects projective objects.
To give the argument explicitly, suppose that $F_A(X,\theta)$ is projective.
For any epimorphism $F\colon(Y,\sigma)\to(Z,\tau)$ and any morphism $H\colon(X,\theta)\to(Z,\tau)$,
the morphism $F_A(F)$ is an epimorphism in $\rep(A)$. Projectivity produces an $A$-linear map
\[ h\colon F_A(X,\theta) \to F_A(Y,\sigma) \]
such that $F_A(F)h=F_A(H)$. By fullness, $h$ is the image of a Kleisli convolution morphism $\widetilde h$.
By faithfulness, $F\widetilde h=H$. Thus, $(X,\theta)$ is projective in $\KCrep(A)$.
The converse follows by applying a quasi-inverse of $F_A$.

Every finite-dimensional projective left $A$-module is a direct summand of $A^n$ for some $n$.
Conversely, every direct summand of $A^n$ is projective.
This proves the result.
\end{proof}

\subsubsection{Injective Kleisli convolution representations}

\begin{definition} \rm
An object $(X,\theta)$ of $\KCrep(A)$ is called an \defines{injective Kleisli convolution representation}
if $F_A(X,\theta)$ is an injective left $A$-module. We denote the full subcategory consisting of all such objects by $\KCrep^{\mathrm{inj}}(A)$.
\end{definition}

Let $D=\Hom_{\kk}(-,\kk)$ be the standard duality. Recall that $D(A_A)$ is an injective cogenerator in $\rep(A)$.

\begin{proposition}\label{prop:algebra-injective}
The equivalence given in Theorem \ref{thm:algebra-equivalence} restricts to an equivalence
\[ \KCrep^{\mathrm{inj}}(A)\simeq\mathsf{inj}(A),\]
where $\mathsf{inj}(A)$ is the category of finite-dimensional injective left $A$-modules.
\end{proposition}

\begin{proof}
The proof is dual to that of Proposition \ref{prop:algebra-projective}.
More precisely, an object is injective if every morphism from a subobject extends to the containing object.
Since $F_A$ and its quasi-inverse preserve monomorphisms, and since they are full and faithful, this extension
property holds for $(X,\theta)$ if and only if it holds for $F_A(X,\theta)$.

Every finite-dimensional injective left $A$-module is a direct summand of a finite direct sum of copies of the injective cogenerator $D(A_A)$.
Conversely, finite direct sums and direct summands of injective modules are injective.
Transporting these statements through $F_A$ gives the result.
\end{proof}


\subsubsection{Simple Kleisli convolution representations}

\begin{definition} \rm
A nonzero object $(X,\theta)$ of $\KCrep(A)$ is called a \defines{simple Kleisli convolution representation} if $F_A(X,\theta)$ is a simple left $A$-module.
\end{definition}

\begin{proposition}\label{prop:algebra-simple}
The equivalence given in Theorem \ref{thm:algebra-equivalence} induces a bijection
between the isomorphism classes of simple objects of $\KCrep(A)$ and the isomorphism classes of simple left $A$-modules.
Equivalently, a nonzero object $(X,\theta)$ is simple if and only if the only subspaces $W\subseteq\kk^{\oplus X}$ satisfying $\Phi_X(\theta(a)) W\subseteq W$ for all $a\in A$ are $0$ and $\kk^{\oplus X}$.
\end{proposition}

\begin{proof}
Subobjects of $(X,\theta)$ correspond, under the equivalence given in Theorem \ref{thm:algebra-equivalence}, to $A$-submodules of $\kk^{\oplus X}$.
A vector subspace $W\subseteq\kk^{\oplus X}$ is an $A$-submodule precisely when it is invariant under every matrix $\Phi_X(\theta(a))$. Therefore, $(X,\theta)$ has no proper nonzero subobject if and only if $\kk^{\oplus X}$ has no proper nonzero $A$-submodule.
This is exactly the asserted condition.
Every finite-dimensional simple $A$-module has a finite basis.
The essential surjectivity of $F_A$ therefore gives a simple Kleisli convolution representation representing it.
Fullness and faithfulness show that two such objects are isomorphic exactly when their associated $A$-modules are isomorphic.
\end{proof}

If $\kk$ is algebraically closed and $(X,\theta)$ is simple, then Schur's lemma implies
\[ \End_{\KCrep(A)}(X,\theta) \cong\End_A(F_A(X,\theta)) \cong\kk. \]

\section{Descent and the modular isomorphism problem} \label{sect:MIP}

\textsl{In this section, we consider the effect of changing the coefficient field on regular Kleisli convolution representations.
Using Lang's theorem, we give a sufficient condition under which an isomorphism over a finite extension descends to the original finite field.
Throughout this section, all algebras are associative and unital, and all algebra homomorphisms preserve the identity.
For a field extension $\EE/\kk$ and a $\kk$-algebra $A$, write $A_{\EE}={\EE}\otimes_{\kk}A$.}

\subsection{Group algebras as regular Kleisli convolution subalgebras}

Let $G$ be a finite group and let $\kk$ be a field.
For each $g\in G$, we define an element $\rho_{\kk,G}(g)\in\KlHom_{\kk}(G)$, which is a Kleisli arrow $\rho_{\kk,G}(g):\bfone\rightsquigarrow E_G$, by the formula below.
Here, $\bfone$ is a singleton set whose element is written as $\bullet$,
and $E_G = G\otimes G = G\times G$ since the tensor product of objects in Kleisli category $\Kl(\scrM_{\kk})$ is given by the Cartesian product (i.e., $\otimes = \times$). Define
\[ \rho_{\kk,G}(g)(\bullet) = \sum_{x\in G}[(gx,x)].\]

For a field $\FF$, let $B_{\FF}(G)$ denote the $\FF$-linear span of these elements in $\KlHom_{\FF}(G)$, i.e.,
\[
  B_{\FF}(G)
:= \operatorname{span}_{\FF}
   \{\rho_{\FF,G}(g):g\in G\}
 = \bigg\{\sum_{g\in G}a_g\rho_{\FF,G}(g): a_g\in\FF\bigg\}
 \subseteq\KlHom_{\FF}(G).
\]
We have the following lemma.

\begin{lemma}\label{lemm:regular convolution subalgebra}
The space $B_{\kk}(G)$ is a unital subalgebra of $\KlHom_{\kk}(G)$ under Kleisli convolution.
Moreover, the linear extension of $\rho_{\kk,G}$ induces an isomorphism of $\kk$-algebras
\[  \rho_{\kk,G}:\kk G \to B_{\kk}(G). \]
For every field extension $\EE/\kk$, there is a natural isomorphism $\EE\otimes_{\kk}B_{\kk}(G)\cong B_{\EE}(G)$.
\end{lemma}

\begin{proof}
By Lemma \ref{lemm:KlHom-alg}, we have
\[(\rho_{\kk,G}(g)\ast\rho_{\kk,G}(h))(\bullet)
 = \sum_{x\in G}[(ghx,x)]
 = \rho_{\kk,G}(gh)(\bullet)\]
for any $g,h\in G$, and have
\[\rho_{\kk,G}(1_G)=u_G : \bfone \to \semimod_{\kk}(E_G).\]
Then $B_{\kk}(G)$ is a unital subalgebra of $\KlHom_{\kk}(G)$.

Since $\{\rho_{\kk,G}(g):g\in G\}$ spans $B_{\kk}(G)$, the linear extension of $\rho_{\kk,G}$ is surjective.

Next, suppose that $\sum\limits_{g\in G}a_g\rho_{\kk,G}(g)=0$.
Evaluating at $\bullet$, we obtain
$\sum\limits_{g\in G}(a_g\rho_{\kk,G}(g))(\bullet)
= \sum\limits_{g,x\in G}a_g[(gx,x)]=0$.
For each $g\in G$, the coefficient of $[(g,1_G)]$ is $a_g$.
Thus every $a_g$ is zero, and then we get that $\rho_{\kk,G}$ is injective.

Finally, all elements $\rho_{\kk,G}(g)$ form a basis of $B_{\kk}(G)$,
and their multiplication is given by $\rho_{\kk,G}(g)\ast\rho_{\kk,G}(h)=\rho_{\kk,G}(gh)$.
Therefore, the map $\lambda\otimes\rho_{\kk,G}(g)\mapsto \lambda \rho_{\EE,G}(g)$ gives the required isomorphism.
\end{proof}

\begin{definition} \rm
Lemma \ref{lemm:regular convolution subalgebra} shows that $B_{\kk}(G)$ is a finite-dimensional $\kk$-algebra.
We call it a \defines{regular Kleisli convolution subalgebra} of $\KlHom_{\kk}(G)$.
\end{definition}

\begin{remark} \rm
Here, the word ``regular'' comes from the left regular representation of $G$ on $\kk G$.
For each $g\in G$, this representation sends a basis element $x\in G$ to $gx$.
Under the isomorphism $\Phi_G$ in Lemma \ref{lemm:KlHom-alg},
the element $\rho_{\kk,G}(g)$ corresponds precisely to the matrix of this linear map with respect to the basis $G$.
Thus, $B_{\kk}(G)$ is the image of the group algebra under the left regular representation, written in terms of Kleisli convolution.
\end{remark}

\subsection{A descent theorem for finite-dimensional algebras over finite fields}

Let $\kk=\FF_q$ be a finite field and let $\overline{\kk}$ be an algebraic closure of $\kk$.
Let $A$ be a finite-dimensional $\kk$-algebra and assume $\dim_{\kk}A=n$.
After choosing a basis of $A$, every $\overline{\kk}$-algebra automorphism of $A_{\overline{\kk}}$
is represented by an invertible $n\times n$ matrix over $\overline{\kk}$.
Such a matrix corresponds to an algebra automorphism when it preserves multiplication and the identity.
These conditions are polynomial equations in its entries.
To be more precise, choose a basis $e_1,\ldots,e_n$ of $A$
and write $e_ie_j=\sum\limits_{r=1}^n c_{ij}^{r}e_r$ and
$1_A=\sum\limits_{i=1}^n u_i e_i$, where $c_{ij}^{r},~u_i\in\kk$.
Let $(t_{ai})$ be the matrix of an $\overline{\kk}$-linear map $\varphi$,
then we have $\varphi(e_i)=\sum\limits_{a=1}^n t_{ai}e_a$.
Then
\[ \varphi(e_ie_j)=\varphi(e_i)\varphi(e_j) \text{~if and only if~}
  \sum_{\ell=1}^n c_{ij}^{\ell}t_{r\ell}=\sum_{a,b=1}^n t_{ai}t_{bj}c_{ab}^{r}
  ~(1\=< i,j,r\=< n).\]
Similarly,
\begin{center}
  $\varphi(1_A)=1_A$ holds if and only if
$\displaystyle\sum_{i=1}^n u_i t_{ri}=u_r$ for every $1\=< r\=< n$.
\end{center}
These are polynomial equations in the matrix entries $t_{ai}$, with coefficients in $\kk$.
In the case of invertible matrices, their solutions are precisely the $\overline{\kk}$-algebra automorphisms of $A_{\overline{\kk}}$.
These automorphisms (or equivalently, matrices) form an \defines{affine algebraic group},
denoted by $\Aut(A_{\overline{\kk}})$.
We consider its connected components in the Zariski topology,
i.e., the topology whose closed sets are defined by polynomial equations, and set
\[  c(A)=\text{the number of connected components of } \Aut(A_{\overline{\kk}}).\]
Then $c(A)<\infty$ since an affine algebraic variety has only finitely many connected components.
The component containing the identity matrix is called the identity component.

The following proposition follows from Lang's theorem \cite{Lang1956}.
We include its proof to explain how the number $c(A)$ controls the extension degree.

\begin{proposition}\label{prop:finite field descent}
Let $A$ and $B$ be finite-dimensional $\FF_q$-algebras, and let $m$ be a positive integer.
If $\gcd(m,c(A)!)=1$, then $A_{\FF_{q^m}}\cong B_{\FF_{q^m}}$ implies $A\cong B$.
\end{proposition}
\begin{proof}
Write $\FF_q=\kk$ in this proof, and suppose that $A_{\FF_{q^m}}\cong B_{\FF_{q^m}}$.
Then $\dim_{\kk}A=\dim_{\kk}B=n$. Let $\{e_1,\ldots,e_n\}$ and $\{f_1,\ldots,f_n\}$ be two bases of $A$ and $B$ over $\kk$, respectively. Then we have
\[ e_ie_j=\sum_{\ell=1}^n c_{ij}^{\ell}e_\ell, \quad f_af_b=\sum_{r=1}^n d_{ab}^{r}f_r, \]
where all coefficients $c_{ij}^{\ell}$ and $d_{ab}^{r}$ lie in $\kk$.
Write $1_A=\sum\limits_{i=1}^n u_i e_i$ and $1_B=\sum\limits_{r=1}^n v_r f_r$,
where $u_i,v_r\in\kk$.
Consider the set
$\mathfrak I = \operatorname{Iso}_{\overline{\kk}\text{-alg}}(A_{\overline{\kk}},B_{\overline{\kk}})$
of all $\overline{\kk}$-algebra isomorphisms from $A_{\overline{\kk}}$ to $B_{\overline{\kk}}$.
Under the bases of $A$ and $B$ over $\kk$, we regard $\mathfrak I$ as the set of invertible matrices
(preserving multiplication and the identity), with the Zariski topology.
Then, for each $\overline{\kk}$-linear map $\varphi:A_{\overline{\kk}}\to B_{\overline{\kk}}$
with $\varphi(e_i)=\sum\limits_{a=1}^n t_{ai}f_a$, the corresponding polynomial equations
\begin{align}\label{eq:finite field descent 1}
\sum_{\ell=1}^n c_{ij}^{\ell}t_{r\ell} = \sum_{a,b=1}^n t_{ai}t_{bj}d_{ab}^{r}
\end{align}
and
\begin{align}\label{eq:finite field descent 2}
\sum_{i=1}^n u_i t_{ri}=v_r
\end{align}
have coefficients in $\kk$.
Now, assume that $\varphi\in\mathfrak I$ is an isomorphism whose matrix entries belong to $\FF_{q^m}$.
Then for any $a\in \Aut(A_{\overline{\kk}})$, the composition $A_{\overline{\kk}} \xrightarrow{a} A_{\overline{\kk}} \xrightarrow{\varphi} B_{\overline{\kk}}$ induces a map
\begin{align}\label{map:finite field descent}
\Aut(A_{\overline{\kk}}) \to \mathfrak I,\quad a\mapsto\varphi\compos a
\end{align}
which is invertible, and  its inverse sends any $\psi \in \mathfrak I$ to $\varphi^{-1}\compos\psi$.
Since their matrix entries are linear expressions in the entries of the variable matrix,
both maps are continuous in the Zariski topology.
Thus this map provides a bijection between the connected components of $\Aut(A_{\overline{\kk}})$ and the connected components of $\mathfrak I$.
Therefore, by the definition of $c(A)$, $\mathfrak I$ has exactly $c(A)$ connected components.
Let $F:\Mat_n(\overline{\kk})\to\Mat_n(\overline{\kk})$ be defined by $F((t_{ij})_{i,j})=(t_{ij}^{q})_{i,j}$, where $\dim_{\kk}A=\dim_{\kk}B=n$.
Since \eqref{eq:finite field descent 1} and \eqref{eq:finite field descent 2} have coefficients in $\kk=\FF_q$,
and $F$ preserves invertibility, we have that $F$ sends $\mathfrak I$ to itself.
Moreover, $F$ is a homeomorphism in the Zariski topology,
so it permutes the connected components of $\mathfrak I$.
Let $C$ be the component containing $\varphi$.
Since the entries of $\varphi$ belong to $\FF_{q^m}$, we have $F^m(\varphi)=\varphi$, and then we have $F^m(C)=C$.
If the orbit $O_C:=\{F^t(C) : t\in\NN \}$ of $C$ under $F$ has length $d$,
then $d\mid m$ and $1\=< d\=< c(A)$. Thus, $d\mid c(A)!$.
By $\gcd(m,c(A)!)=1$, we have $d=1$, and so $F(C)=C$.
Let $\mathcal E$ be the identity component of $\Aut(A_{\overline{\kk}})$.
Under the map \eqref{map:finite field descent}, the identity automorphism is sent to $\varphi$.
It follows that $\mathcal E$ is mapped onto $C$.
Since $F(\varphi)\in C$, there is an element $a \in\mathcal E$ such that $F(\varphi)=\varphi\compos a$.
The identity component $\mathcal E$ is a connected algebraic group defined over $\kk$.
Lang's theorem \cite{Lang1956} implies that the map
\[
\mathcal E \to \mathcal E, \quad b\mapsto
(b^{-1}\compos F(b):
A_{\overline{\kk}} \xrightarrow{F(b)} A_{\overline{\kk}} \xrightarrow{b^{-1}} A_{\overline{\kk}}
)
\]
is surjective on this group.
Applying this statement to $a$, we can choose $c\in\mathcal E$ with $c^{-1}\compos F(c)=a$.
Take $b=c^{-1}$, then we have $a\compos F(b)=b$. Consequently,
\[ F(\varphi\compos b) = F(\varphi)\compos F(b) = \varphi\compos a\compos F(b) = \varphi\compos b. \]
Every matrix entry of $\varphi\compos b$ is therefore fixed by the $q$-th power map, and hence belongs to $\kk=\FF_q$.
Thus, $\varphi\compos b$ gives an isomorphism $A\cong B$ of $\FF_q$-algebras.
\end{proof}

\subsection{Application to the modular isomorphism problem}

Let $p$ be a prime. For a finite $p$-group $G$, put $c_G=c(B_{\FF_p}(G))$.
By Lemma \ref{lemm:regular convolution subalgebra}, this is also the number of connected components of the algebraic group
$\Aut_{\overline{\FF}_p\text{-alg}} (\overline{\FF}_pG)$.

\begin{corollary}\label{coro:group algebra finite field descent}
Let $G$ and $H$ be finite $p$-groups, and let $m$ be a positive integer satisfying $\gcd(m,c_G!)=1$. Then
\[   \FF_{p^m}G\cong\FF_{p^m}H \text{ implies } \FF_pG\cong\FF_pH.\]
In particular, if $\FF_pG\cong\FF_pH$ implies $G\cong H$ for every finite $p$-group $H$,
then
\[ \FF_{p^m}G\cong\FF_{p^m}H \text{ also implies } G\cong H \]
for every finite $p$-group $H$ whenever $\gcd(m,c_G!)=1$.
\end{corollary}

\begin{proof}
By Lemma \ref{lemm:regular convolution subalgebra},
the isomorphism $\FF_{p^m}G\cong\FF_{p^m}H$ implies
\[ \FF_{p^m}\otimes_{\FF_p}B_{\FF_p}(G) \cong \FF_{p^m}\otimes_{\FF_p}B_{\FF_p}(H).\]
Proposition \ref{prop:finite field descent} yields $B_{\FF_p}(G)\cong B_{\FF_p}(H)$.
Applying Lemma \ref{lemm:regular convolution subalgebra} again, we obtain $\FF_pG\cong\FF_pH$.
The last assertion follows immediately.
\end{proof}

For a finite $p$-group $G$, let
\[ I_G :=\bigg\{\sum_{g\in G}a_gg\in\FF_pG: \sum_{g\in G}a_g=0\bigg\}
  =\operatorname{span}_{\FF_p}\{g-1:g\in G\}, \]
be the \defines{augmentation ideal} of $\FF_pG$,
namely the kernel of the map $\sum\limits_{g\in G}a_gg\mapsto\sum\limits_{g\in G}a_g$, see \cite{Passman1977} for details.
The third dimension subgroup is
\[ D_3(G) := \{g\in G:g-1\in I_G^3\}. \]
Passi and Sehgal proved that a finite $p$-group with $D_3(G)=1$ is determined by its group algebra over $\FF_p$ \cite{PSl1972}.
Whether the same conclusion holds over every field of characteristic $p$ is asked in \cite[Question 2.10]{MS2025}. Corollary \ref{coro:group algebra finite field descent} gives the following consequence.

\begin{theorem}\label{coro:third dimension finite extensions}
Let $G$ be a finite $p$-group with $D_3(G)=1$.
Let $m$ be a positive integer such that $\gcd(m,c_G!)=1$.
Then, for every finite $p$-group $H$,
\[  \FF_{p^m}G\cong\FF_{p^m}H \text{~implies~} G\cong H. \]
In particular, this holds whenever $m$ is a prime greater than $c_G$.
\end{theorem}

\begin{proof}
By Corollary \ref{coro:group algebra finite field descent},
we have $\FF_pG\cong\FF_pH$.
Since $D_3(G)=1$, the result of Passi and Sehgal \cite{PSl1972} gives $G\cong H$.
If $m$ is a prime greater than $c_G$, then $\gcd(m,c_G!)=1$.
\end{proof}

The descent argument in Proposition \ref{prop:finite field descent} applies to arbitrary finite-dimensional algebras over finite fields.
Here, regular Kleisli convolution representations give a concrete realization of the algebras to which it is applied.
The argument does not yet give an explicit bound for $c_G$.
To obtain a more explicit range of coefficient fields,
one needs to determine or bound the number of connected components of the automorphism algebraic group of $B_{\FF_p}(G)$.
Thus Theorem \ref{coro:third dimension finite extensions} does not settle the corresponding question over all fields.

%

\paragraph{Authors' Contributions}



\paragraph{Competing Interests}

The author declare that they have no conflicts of interest as defined by the journal, nor any other interests that could be perceived as influencing the results presented in this paper.

\paragraph{Data availability}
Data sharing not applicable to this article as no datasets were generated or analysed during the current study

\paragraph{Ethical Approval}

This article does not require ethical approval.

\paragraph{Fundings}
Yu-Zhe Liu is supported by
the National Natural Science Foundation of China (Grant Nos. 12401042 and 12561008),
the Guizhou Provincial Key Laboratory of Applied Mathematics and Computing Power \& Algorithms (No. Qiankehe Platform ZSYS[2025]039),
the Science and Technology Foundation of the Guizhou S~\&~T Department (Grant Nos. VZD[2026]001, ZD[2025]085 and ZK[2024]YiBan066),
and Scientific Research Foundation of Guizhou University (Grant No. [2023]16).



\newcommand{\etalchar}[1]{$^{#1}$}

\end{document}